\RequirePackage{etoolbox}
\csdef{input@path}{{style/}{graphics/}}
\documentclass[aap,MSNbibl,seceqn,dvips]{arximspdf}
\usepackage{mathbh}
\usepackage{graphicx}

\doi{10.1214/14-AAP1006} 
\volume{25}
\issue{2}
\pubyear{2015}
\firstpage{632}
\lastpage{662}

\makeatletter
\newcommand{\rrvert}{\vert}
\newcommand{\llvert}{\vert}
\newcommand{\underset}[2]{\mathop{#2}\limits_{#1}}
\newcommand{\one}{\mathbh{1}}
\newcommand{\un}{\mathbh{1}}
\newcommand{\R}{\mathbb R}
\renewcommand{\P}{\mathbb P}
\newcommand{\Z}{\mathbb Z}
\newcommand{\N}{\mathbb N}
\newcommand{\T}{\mathbb T} 

\newcommand{\X}{\mathcal X}
\newcommand{\Y}{\mathcal X} 
\newcommand{\setZ}{\mathcal Z}
\newcommand{\V}{\mathcal V}
\newcommand{\W}{\mathcal W}
\newcommand{\outdeg}{Y}

\newproclaim{rem}{Remark}
\newtheorem{theorem}{Theorem}
\newtheorem{cor}[theorem]{Corollary}
\newtheorem{lemma}[theorem]{Lemma}
\newtheorem{prop}[theorem]{Proposition}
\makeatother

\begin{document}
\begin{frontmatter}

\title{Spatial preferential attachment networks:\\ Power laws and clustering coefficients\thanksref{T1}}
\runtitle{Spatial preferential attachment networks}

\begin{aug}
\author[A]{\fnms{Emmanuel} \snm{Jacob}\ead[label=e1]{emmanuel.jacob@ens-lyon.fr}}
\and
\author[B]{\fnms{Peter} \snm{M\"orters}\corref{}\ead[label=e2]{p.morters@bath.ac.uk}}
\runauthor{E. Jacob and P. M\"orters}
\affiliation{ENS Lyon and University of Bath}
\address[A]{ENS Lyon\\
46 all\'ee d'Italie\\
69007 Lyon\\
France\\
\printead{e1}}
\address[B]{Department of Mathematical Sciences\\
University of Bath\\
Bath BA2 7AY\\
United Kingdom\\
\printead{e2}}
\end{aug}
\thankstext{T1}{Supported by the \emph{European Science Foundation} through
the research network \emph{Random Geometry of Large Interacting Systems
and Statistical Physics} (\emph{RGLIS}).}

\received{\smonth{10} \syear{2012}}
\revised{\smonth{5} \syear{2013}}

%
\begin{abstract}
We define a class of growing networks in which new nodes are given a
spatial position and
are connected to existing nodes with a probability mechanism favoring
short distances
and high degrees. The competition of preferential attachment and
spatial clustering
gives this model a range of interesting properties. Empirical degree
distributions converge
to a limit law, which can be a power law with any exponent~$\tau>2$.
The average clustering
coefficient of the networks converges to a positive limit. Finally, a
phase transition occurs in
the global clustering coefficients and empirical distribution of edge lengths
when the power-law exponent crosses the critical value~$\tau=3$. Our
main tool in the proof of
these results is a general weak law of large numbers in the spirit of
Penrose and Yukich.
\end{abstract}

%

\pdfkeywords{05C80, 60C05, 90B15, Scale-free network, Barabasi-Albert model, preferential attachment, dynamical random graph,
geometric random graph, power law, degree distribution, edge length distribution, clustering coefficient}
\end{frontmatter}

\setcounter{footnote}{1}

\section{Introduction}\label{sec1}

Many of the phenomena in the complex world in which we live have a
rough description as a large network of interacting
components. It is therefore a fundamental problem to derive the global
structure of such networks
from basic local principles. A well-established principle is the \emph{preferential attachment paradigm} which suggests
that networks are built by adding nodes and links successively, in such
a way that new nodes prefer to be connected to
existing nodes if they have a high degree~\cite{BarabasiAlbert}. The
preferential attachment paradigm offers, for example, a credible explanation
of the observation that many real networks have degree distributions
following a \emph{power law behavior}. On the global
scale preferential attachment networks are \emph{robust} under random
attack if the power law exponent
is sufficiently small, and have logarithmic or doubly logarithmic
diameters depending on the power law exponent. These
features, together with a reasonable degree of mathematical
tractability, have all contributed to the enormous popularity of these
models.

Among the many criticisms directed at preferential attachment models is
a significant deviation of their local structure from that observed in
real networks. In preferential attachment models, the neighborhoods of
typical nodes have a tree-like topology~\cite{DMgiant,weak},
which is a crucial feature for their global analysis, but is not in
line with the behavior of many real world networks.
The most popular quantities used to measure the local clustering of
networks are the \emph{clustering coefficients}, which are measured to be
positive in most real networks, but which invariably vanish in
preferential attachment models that do not incorporate further
effects~\mbox{\cite{ABsurvey,BRsurvey}}.
A~possible reason for the clustering of real networks is the presence
of a hidden variable assigned to the nodes, such that similarity
of values is a further incentive to form links.
Several authors have therefore proposed models combining preferential
attachment with spatial features in order to address the weaknesses of pure
preferential attachment. Among the mathematically sound attempts in
this direction are the papers of Flaxman, Frieze and Vera \cite{Flaxman,Flaxman2},
Jordan~\cite{Jordan}, Jordan and Wade~\cite{Jordan-Wade}, Aiello et
al.~\cite{Aiello} and Cooper, Frieze and Pra\l{}at \cite{Cooper}.
These papers show that combining preferential attachment and spatial
dependence can retain the global power law behavior while changing the local
topology of the network, for example, by showing that the resulting
graphs have small separators~\cite{Flaxman,Flaxman2},
but none of them discusses clustering systematically by analyzing the
clustering coefficients.

In this paper we propose a natural model of a network in which the
preferential attachment paradigm is modulated by spatial proximity. Our
model is
a generalization and variant of the one introduced in Aiello et
al.~\cite{Aiello}.
The model is best described as a growing network in continuous
time. New nodes are born according to a Poisson process of rate one and
placed uniformly on the one-dimensional torus of length one. 
A node born at time $t$ is connected by an ordered edge to each
existing node independently with a probability $\varphi(t\rho/f(d))$
where $d$ is the indegree of the older node
at time $t$, and $\rho$ is the distance of the nodes. The decreasing
\emph{profile function} $\varphi\dvtx [0,\infty)\to[0,1]$ and
increasing \emph{attachment rule}
$f\dvtx \N\cup\{0\} \to(0,\infty)$ are the parameters of the model.
Loosely speaking, the fact that the time~$t$ and the spatial
distance~$\rho$ appear as a product in the connection probability
ensures that
the probability that new nodes connect to their spatially nearest
neighbors, which typically are distance $1/t$ away and have bounded indegree,
does not go to zero or one. This is necessary to balance the spatial
and preferential attachment effects in our model. We show that this
modification of the original idea of preferential attachment preserves
the power
law behavior of existing preferential attachment models while
significantly changing the local topology leading to a positive average
clustering coefficient.
We also observe interesting phase transitions in the behavior of the
global clustering coefficient and the empirical edge length distribution.

Our analysis of this model is using methods developed originally for
the study of random geometric graphs; see Penrose and Yukich~\cite
{YukichPenrose} for a seminal paper
in this area and~\cite{Penrose} for an exhibition. This approach is new
in the context of preferential attachment and quite different from the
established route to study dynamical random graph models, which is
based on the use of differential equations to study the evolution of
expected quantities and concentration inequalities to relate them to
the empirical quantities. By contrast, our analysis is based on a
rescaling which transforms the growth in time into a growth in space.
This transformation stabilizes the neighborhoods of a typical vertex
and allows us to observe convergence of the local neighborhoods of
typical vertices in the graph to an infinite graph. This infinite
graph, which is not a tree, is locally finite and can be described by
means of a Poisson point process. We establish a weak law of large
numbers, similar to the one given in~\cite{YukichPenrose}, which allows
us to deduce convergence results for a large class of functionals of
the graph. Some further work is required to show that certain rare
effects, like vertices having a very high degree or being linked to
distant vertices, do not affect our functionals.%

The paper is organized as follows. In Section~\ref{sec2} we present the model.
The main results concerning the degree distribution, the clustering
coefficients and the edge length distribution,
are stated in Section~\ref{sec3}. In Section~\ref{sec4} we describe the general method
and main tools developed for the study of the network.
Section~\ref{sec5} completes the proofs of our main results and, finally,
Section~\ref{sec6} briefly discusses some variants and further developments.

\section{The model}\label{sec2}

Write $\T_1$ for the one-dimensional torus of length 1 represented as
$\R/\Z$ endowed with the usual distance. Let $\X$ denote a Poisson
point process of unit intensity on $\T_1\times(0,\infty)$. A point
$\mathbf{x}
=(x,s)$ in $\X$ is a vertex $\mathbf{x}$, born at time~$s$ and placed at
position $x$.
Observe that, almost surely, two points of $\X$ neither have the same
birth time nor the same position.
We say that $(x,s)$ is \emph{older} than $(y,t)$ if $s<t$. An edge
is always oriented from the younger to the older vertex. For $t> 0$,
write $\X_t$ for $\X\cap(\T_1\times[0,t])$, the set of vertices
already born at time $t$. We construct a growing
sequence of graphs $(G_t)_{t>0}$, starting from the empty graph, and
adding successively the vertices in $\X$ when they are born (so that
the vertex set of $G_t$ is $\X_t$), and connecting them to some of the
older vertices. The rule is as follows:

\subsection*{Construction rule}
Given the graph $G_{t-}$ and $\mathbf
{y}=(y,t)\in\X
$, we add the vertex~$\mathbf{y}$ and, independently for each vertex
$\mathbf{x}$ in $G_{t-}$,
we insert the edge $(\mathbf{y}, \mathbf{x})$, independently of $\X
$, with probability
%
%
\begin{equation}
\label{constructionrule} \varphi\biggl(\frac{t d(\mathbf{x},\mathbf
{y})} {f(Z_\mathbf
{x}(t-))} \biggr).
\end{equation}
The resulting graph is denoted by $G_t$.

Here the following definitions and conventions apply:
\begin{longlist}[(3)]
\item[(1)] $d(\mathbf{x}, \mathbf{y})$ denotes the length of the edge
$(\mathbf{y}, \mathbf{x})$, which is
the usual distance in~$\T_1$ (for which, by a minor abuse of notation
we also use the notation~$d$) between the spatial positions of the
vertices $\mathbf{x}$ and~$\mathbf{y}$.
\item[(2)] $\varphi\dvtx [0,\infty)\to[0,1]$ is the \emph{profile function}.
It is supposed to be nonincreasing and of total integral $1/2$.
Informally, it describes the spatial dependence of the probability that
the newborn vertex $\mathbf{y}$ is linked to the existing vertex
$\mathbf{x}$.
\item[(3)] $Z_\mathbf{x}(t-)$ [resp., $Z_\mathbf{x}(t)$] denotes the
indegree of vertex $\mathbf{x}$
at time $t-$ (resp.,~$t$), that is, the total number of incoming edges
for the vertex $\mathbf{x}$ in $G_{t-}$ (resp., $G_t$). Similarly, we
denote by
$\outdeg_\mathbf{y}$ the outdegree of vertex $\mathbf{y}$, which
remains the same at
all times $u\ge t$.

\item[(4)] $f\dvtx \N\cup\{0\} \to(0,\infty)$ is the \emph{attachment
rule}. It is supposed to be nondecreasing. Informally, $f(k)$
quantifies the preferential ``strength'' of a vertex of current
indegree~$k$, or likelihood of attracting new links. We assume that the
attachment rule $f$ has an asymptotic slope
\[
\gamma:= \lim_{k\to\infty} \frac{f(k)}{k} \in(0,1).
\]
\end{longlist}

Note that, for any $r>0$, the profile function~$\varphi$ and attachment
rule~$f$ together define the same model as the
profile function~$x\mapsto\varphi(rx)$ and the attachment rule~$k
\mapsto rf(k)$.
The normalization convention $\int\varphi=\frac{1}2$, which will always
be assumed for convenience, represents therefore no loss of~generality.\looseness=-1

Whereas in classical preferential attachment the linking probability
itself is multiplied by the preferential attachment factor $f(Z_\mathbf
{x}(t-))$,
in our spatial setup this factor enters as the spatial expansion of the
\emph{influence profile} around the vertex~$\mathbf{x}=(x,s)$ at time~$t$,
which is
described by the function
\[
y \mapsto x+\varphi\biggl(\frac{t d(x,y)} {f(Z_\mathbf{x}(t-))} \biggr).
\]
The probability of connecting a new vertex~$(y,t)$ to an old one is
given by the value of the influence
profile around the old vertex at the position~$y$ of the new one. In
the important special case of the profile function
$\varphi(r)= \un{\{r<\frac{1}2\}}$, which only takes the values zero or
one, this decision is not random.
In this case a
vertex~$\mathbf{x}$ is linked to a new vertex born at time $t$ if and
only if
their positions are within distance ${f(Z_\mathbf{x}(t-))}/(2t)$. In other
words, every
vertex~$\mathbf{x}$ is surrounded by an \emph{influence region}, a
ball of
time-dependent radius ${f(Z_\mathbf{x}(t-))}/(2t)$, and a new vertex
is linked
to all older vertices in
whose influence regions it falls at the time of its birth. This special
case already reveals the complexity and interest of the model, and the
reader is encouraged
to first figure out its behavior.

The model introduced by Aiello et al.~\cite{Aiello} and further studied
by Cooper, Frieze and Pra\l{}at \cite{Cooper} and by Janssen, Pra\l{}at and
Wilson~\cite{JPW} is essentially the\vadjust{\goodbreak} same model for the special case
that the attachment rule is of the form $f(k)=A_1k+A_2$ and the profile
function is of the form $\varphi(x)= p \un{\{x<1/(2p)\}}$.
Small differences are that they work in discrete rather than continuous
time, and allow for spaces more general than~$\T_1$, but these
differences are inessential for the purposes of this paper; see also
our comments in Section~\ref{sec6}.

Recall the definition of the asymptotic slope~$\gamma$ of the
attachment function from~(4). As $\gamma>0$ this means that $f$ is
asymptotically linear, and this is known, in nonspatial preferential
attachment models, to lead to scale-free networks with power law
exponent $\tau=1+\frac{1} \gamma$.%

We now illustrate the connection between nonspatial preferential and
spatial attachment models.
Suppose the graph $G_{t-}$ is given, and a vertex is born at time $t$,
but we do not know its position, which is therefore uniform on $\T_1$.
Then, for each vertex $\mathbf{x}=(x,s) \in G_{t-}$, the probability
that it is
linked to the newborn vertex is equal to
\[
\int_{\T_1} \varphi\biggl(\frac{t d(x,y)} {f(Z_\mathbf
{x}(t-))} \biggr)\,\mathrm{d}y = \frac{f(Z_\mathbf{x}(t-))} t 2 \int_{0}^{{t}/
({f(Z_\mathbf
{x}(t-))})}
\varphi(y)\,\mathrm{d}y.
\]
As a consequence, the process $(Z_\mathbf{x}(t))_{t\ge s}$ is a
time-inhomogeneous pure birth process, starting from 0 and jumping at
time $t$ from state $k$ to state $k+1$ with intensity
\[
\frac{f(k)} t 2 \int_{0}^{{t}/({f(k)})} \varphi(x)\,\mathrm{d}x.
\]
This quantity is bounded by $f(k)/t$. As the pure birth process
$(Z_\mathbf{x}
(t))_{t\ge s}$ grows roughly like $t^{\gamma}$ (see Lemma~\ref
{polynomialgrowth} for a precise statement),
the normalization of $\varphi$ makes this bound asymptotically sharp.
Hence the jumping intensity of our process is the same as in the
classical Barab\'asi--Albert model of preferential attachment \cite{BarabasiAlbert,Toth}, or its variant studied by Dereich and M\"
orters~\cite{DMdegrees,DMconcave,DMgiant}. Not surprisingly, our
spatial model exhibits the same limiting indegree distribution.

However, as soon as one deepens the study of the graph further than the
first moment calculations, the essential difference with the
nonspatial models appears. The presence of edges is now strongly
correlated through the spatial positions of the vertices. These strong
correlations both make the model much harder to study and allow the
network to enjoy interesting clustering properties. These are the main
concerns of this paper and will be described in the next section. We
will henceforth use the common notation $g=o(h)$ to indicate that $g/h$
converges to zero, $g \asymp h$ if $g/h$ is bounded from zero and infinity
and $g\sim h$ to indicate that $g/h$ converges to one.

\section{Main results}\label{sec3}

\subsection{Indegree distribution}\label{sec3.1}
While the indegree of a given vertex grows indefinitely with the size
of the network, the \emph{mean indegree} in the graph $G_t$\vadjust{\goodbreak} converges
to a limiting distribution with polynomial decay. More precisely, for
$t>0$ such that $\X_t$ is nonempty, denote by $\mu_t$ the law of the
indegree of a randomly (and uniformly) chosen vertex in the graph
$G_t$, or \emph{empirical indegree distribution}.
More formally, the empirical indegree distribution is the random
measure on $\N\cup\{0\}$, which gives to each $k \in\N\cup\{0\}$
the weight
\[
\mu_t(k) = \frac{1} {|\X_t|} \sum_{\mathbf{x}\in\X_t}
\one{\bigl\{ Z_\mathbf{x}(t)=k\bigr\}},
\]
if $\X_t\neq\varnothing$ and $\mu_t(k)=\one\{k=0\}$ otherwise.
We introduce the probability measure~$\mu$, determined by its weights
%
%
\begin{equation}
\mu(k)= \frac{1} {1+f(k)} \prod_{l=0}^{k-1}
\frac{f(l)} {1+f(l)}.
\end{equation}
For any measure $\lambda$ on $\N\cup\{0\}$ and any function
$g\dvtx \N
\cup\{0\} \to[0,\infty)$, we write $\langle\lambda,g\rangle$ for the
expectation of $g$
under the law $\lambda$, or $\sum_{k\ge0} \lambda(k) g(k)$. The
following theorem states a convergence result for the empirical
indegree distribution $\mu_t$ to the probability measure $\mu$, which
we call limiting indegree distribution. This result implies, in
particular, convergence in probability, in the total variation norm.

%
%
\begin{theorem} \label{theoconvergenceindegree}
For any nondecreasing function $g\dvtx \N\cup\{0\} \to[0,\infty)$
satisfying $\langle\mu,g^p\rangle<\infty$ for some $p>1$, the following
limit holds:
\[
\langle\mu_t,g\rangle\longrightarrow\langle\mu,g\rangle,
\]
in probability, when $t\to\infty$.
\end{theorem}

%
%
\begin{rem}
The convergence in the theorem still holds for any function $g$, not
necessarily positive or monotonous, but with $g(k)=o(k^{ \delta})$ for
some $\delta<1/\gamma$.

It is easy to check that, the limiting distribution $\mu$ satisfies
\[
\mu(k) = k^{-(1+ (1/\gamma))+ o(1)}\qquad\mbox{as } k\uparrow
\infty,
\]
which highlights the scale-free property of the network with exponent
$\tau=1+1/\gamma$.
In the particular case of a linear attachment rule $f(k)=\gamma k +
\beta$, with $\gamma\in(0,1)$ and $\beta>0$, we have
\[
\mu(k) = \frac{1} \gamma\frac{\Gamma(k+(\beta/\gamma)) \Gamma((\beta+1)/ \gamma)} {
\Gamma(k+(({\beta+\gamma+1})/\gamma)) \Gamma(\beta/\gamma)} \sim
\frac{\Gamma(({\beta+1})/ \gamma)} {\gamma\Gamma(\beta/
\gamma)}
k^{-\tau} \qquad\mbox{as } k\uparrow\infty,
\]
a result that has already been obtained for their variant of the model
in Theorem~1.1 of Aiello et al.~\cite{Aiello} by a completely different
technique of proof.
\end{rem}

Our result shows that under our normalization convention, the profile
function has no influence on the degree
distribution. Note, however, that in the presence of spatial dependence
the normalization of the profile function
typically enforces a significant change to the attachment rule. As an
example, we look at the case when the
vertex $\mathbf{y}$ born at time $t$ connects to vertex $\mathbf{x}$
with probability
\[
\biggl(\frac{f(Z_\mathbf{x}(t-))}{t^\alpha d(\mathbf{x},\mathbf
{y})^\alpha} \biggr) \wedge1,
\]
for $\alpha>1$, where $a\wedge b$ denotes the minimum of $a$ and $b$.
In our setup, this must correspond to the normalized profile function
$\varphi(r):=(\frac{2\alpha}{\alpha-1} r)^{-\alpha} \wedge1$
and the attachment rule $f'(k):= \frac{2\alpha}{\alpha-1}
f^{1/\alpha
}(k)$. Thus if $f^{1/\alpha}$ is approximately linear with slope~$\gamma
$, the
resulting power law exponent is $\tau=1+\frac{\alpha-1}{2\gamma
\alpha}$.

\subsection{Outdegree distribution}\label{sec3.2}
In the original preferential attachment model of Barab\'asi and Albert,
the outdegree is constant. In the model variant of Dereich and M\"
orters, it is asymptotically Poisson, therefore it is light-tailed,
which implies that it is not relevant in the study of the tail of the
degree distribution. In our model, the limiting outdegree distribution
is not Poisson, and we could not find a closed formula defining it.
Still, we prove that it is light-tailed.

Denote by $\nu_t$ the empirical outdegree distribution in the graph
$G_t$, defined by its weights
\[
\nu_t(k) = \frac{1} {|\X_t|} \sum_{\mathbf{x}\in\X_t}
\one{\{ \outdeg_\mathbf{x}=k\}},
\]
if $\X_t\neq\varnothing$ and $\nu_t(k)=\one\{k=0\}$ otherwise. The
following theorem holds:

%
%
\begin{theorem} \label{theooutdegree}
There exists a probability measure $\nu$ on $\N\cup\{0\}$ such that:

\begin{longlist}[(2)]
\item[(1)] For any function $g\dvtx \N\cup\{0\} \to\R$ satisfying
$g(k)=o(e^{k^\delta})$ for some $0<\delta<1-\gamma$,
we have
\[
\langle\nu_t,g\rangle\longrightarrow\langle\nu,g\rangle,
\]
in probability, when $t\to\infty$.
\item[(2)] The measure $\nu$ is light-tailed in the following sense: for any
$0<\delta<1-\gamma$, we have
\[
\nu\bigl([k,+\infty)\bigr) = o\bigl(e^{-k^\delta}\bigr).
\]
\end{longlist}
\end{theorem}

The limiting outdegree distribution $\nu$ is implicitly defined [see
formula (\ref{limitingoutdegreeformula}) below],
but it is not easy to compute explicitly. Moreover, it is not hard to
see from our proofs that the indegree and the outdegree of
a randomly chosen vertex are asymptotically independent and hence the
limiting total degree distribution is the
convolution~$\mu\ast\nu$.

\subsection{Clustering}\label{sec3.3}
We now define the clustering coefficients for a finite simple graph
$G=(V,E)$ with unoriented edges,
forgetting the orientation of edges in the case of an oriented graph.
A subgraph of $G$ containing exactly three distinct vertices and the
three edges linking them is called a \emph{triangle}. A subgraph of the
form $(\{\mathbf{x}, \mathbf{y}, \mathbf{z}\}, \{\{\mathbf{x},
\mathbf{y}\}, \{\mathbf{x},\mathbf{z}\}\})$ is called an \emph
{open triangle with tip $\mathbf{x}$}. In other words, an open
triangle with
tip $\mathbf{x}$ consists of the vertex $\mathbf{x}$ and two of its
neighbors $\mathbf{y}$~and~$\mathbf{z}$, which
themselves could either be connected 
and hence form a triangle in $G$, or not. Note that every triangle in
$G$ contributes three open triangles.

The \emph{global clustering coefficient} of $G$ is defined as
\[
c^{\mathrm{glob}}(G): = 3 \frac{\mbox{Number of triangles included
in }G} {\mbox{Number of open triangles included in }G},
\]
if there is at least one open triangle in the graph, and $ c^{\mathrm
{glob}}(G)=0$ otherwise. Note that always $ c^{\mathrm{glob}}(G)\in[0,1]$.
The local clustering coefficient of $G$ at a vertex $\mathbf{x}$ with
degree at
least two is defined by
\[
c_{\mathbf{x}}^{\mathrm{loc}}(G): = \frac{\mbox{Number of triangles
included in
$G$ containing vertex } \mathbf{x}} {\mbox{Number of open triangles
with tip } \mathbf{x}
\mbox{ included in }G},
\]
which is also an element of $[0,1]$. Finally, the \emph{average
clustering coefficient} is defined as
\[
c^{\mathrm{av}}(G):= \frac{1} {|V_2|} \sum_{\mathbf{x}\in V_2}
c_{\mathbf{x}}^{\mathrm{loc}}(G),
\]
if the set $V_2\subset V$ of vertices with degree at least two in $G$
is not empty, and as $ c^{\mathrm{av}}(G):=0$ otherwise.

%
%
\begin{theorem}\label{theoclustering}
(1) \emph{Average clustering coefficient}:\\
There exists a strictly positive number $c^{\mathrm
{av}}_\infty$ such that
\[
c^{\mathrm{av}}(G_t) \longrightarrow c^{\mathrm{av}}_\infty
\]
in probability, as $t\to\infty$.

(2)~\emph{Global clustering coefficient}:
\begin{enumerate}[(a)]
\item[(a)] There exists a nonnegative number $c^{\mathrm{glob}}_\infty$
such that
\[
c^{\mathrm{glob}}(G_t) \longrightarrow c^{\mathrm{glob}}_\infty
\]
in probability, as $t\to\infty$.
\item[(b)] The global clustering coefficient $c^{\mathrm{glob}}_\infty$ is
positive if and only if $\sum k^2 \mu(k)<\infty$.
\end{enumerate}
\end{theorem}

%
%
\begin{rem}
Our proofs allow us to write $c^{\mathrm{glob}}_\infty$ and
$c^{\mathrm{av}}_\infty$ explicitly as
multiple integrals over the network parameters.
\end{rem}

%
%
\begin{rem}
The precise criterion given in Theorem~\ref{theoclustering}(2b)
implies that
$c^{\mathrm{glob}}_\infty>0$ if $\gamma<\frac{1}2$, and
$c^{\mathrm{glob}}_\infty=0$ if $\gamma>\frac{1}2$.
Hence the phase transition in the global clustering coefficient occurs
when the power law exponent crosses
the critical value $\tau=3$.
\end{rem}

%
%
\begin{rem}
The global and average clustering coefficients
have the following probabilistic interpretation:
\begin{itemize}
\item Pick a vertex uniformly at random and condition on the event
that this vertex has degree at least two.
Pick two of its neighbors, uniformly at random. Then the probability
that these two vertices are linked is equal to $c^{\mathrm{av}}(G)$.
\item Pick two edges sharing a vertex, uniformly from all such pairs of
edges in the graph. Then the probability that the two other
vertices bounding the edges are connected is equal to $c^{\mathrm{glob}}(G)$.
\end{itemize}
\end{rem}

Here is an informal discussion of the clustering phenomenon. For a
randomly chosen vertex, both the number of open triangles with tip in that
vertex as well as the number of triangles containing it converge to a
finite random variable. The ratio of these variables determines
the average clustering coefficient, which therefore is always positive.
To understand the phase transition in the behavior of the
global clustering coefficient, first note that, as the outdegree
distribution is always light-tailed, new vertices
typically generate a bounded number of triangles and hence the number
of triangles in the network grows linearly in time.
If $\sum k^2\mu(k)<\infty$ the average number of open triangles per
vertex is finite, and so the number of open triangles
also grows linearly in time, and the global clustering coefficient is
positive. However, if this sum is infinite, the total number of open
triangles has superlinear growth, which is enough to guarantee that the
global clustering coefficient vanishes. In this case,
the tip of a randomly chosen open triangle is typically a very old
vertex with a high degree. This is best seen in the case $\gamma>\frac{1}2$,
in which the degree of the first born vertex at time $t$ is of order
$t^\gamma$, so that this vertex alone gives rise to a superlinear
number $t^{2\gamma}$ of open triangles.
Observe that these effects match the structure of real networks. For
example, if you pick a webpage at random, and click on two hyperlinks,
it is likely that the two pages you get have actually a direct
hyperlink. Now, if you pick two webpages which both have a hyperlink to
the Google homepage, it is not likely that these two pages have a
direct link.

\subsection{Edge length distribution}\label{sec3.4}

In the graph $G_t$, we could hope that a typical edge connects two
vertices with birth times of order $t$ and degrees of order one. We
would then expect from the construction rule (\ref{constructionrule})
that its length is of order $1/t$. This description is actually always
valid within our range of parameters (it would be false for $\gamma\ge
1$), and explains the rescaling below.

Write $E(G_t)$ for the set of the edges of the graph $G_t$. Define
$\lambda$, the (rescaled) empirical edge length distribution, by
\[
\lambda_t = \frac{1} {|E(G_t)|} \sum_{(\mathbf{x}, \mathbf{y}) \in
E(G_t)}
\delta_{t d(\mathbf{x}, \mathbf{y})},
\]
if $E(G_t)\neq\varnothing$, and $ \lambda_t =\delta_0$ otherwise,
where $\delta_u$ is the Dirac measure giving mass one to $\{u\}$.

%
%
\begin{theorem}\label{theoedgelength}
There exists a probability distribution $\lambda$ on the real line
such that:
\begin{longlist}[(2)]
\item[(1)] For every continuous and bounded $g\dvtx [0,\infty) \to\R$ we have
\[
\langle\lambda_t,g\rangle\longrightarrow\langle\lambda,g\rangle,
\]
in probability, when $t\to\infty$.
\item[(2)]
Suppose that there exists $\delta>1$ such that the profile function satisfies
$ \varphi(x) \asymp1 \wedge x^{-\delta}$. Then
\[
\lambda\bigl([K, +\infty)\bigr) \asymp1\wedge K^{-\eta},
\]
where $\eta\in(0,1]$ is the smallest of the three constants $1$,
$\frac{1} \gamma- 1$ and $\delta-1$.
\end{longlist}
\end{theorem}

The heavy tails of the empirical edge length distribution highlight the
nature of our networks as \emph{small worlds}.
Observe that the distribution $\lambda$ never has a first moment,
implying that the \emph{mean edge length} is of larger
order than $1/t$. As the profile function $\varphi$ is integrable, if
it decays polynomially, it must be of order $x^{-\delta}$
for some $\delta>1$. If $\delta\ge2$, then the profile function does
not influence the decay rate of the tail of the limiting
edge length distribution. This stays true if $\varphi$ is any function
satisfying $\int v \varphi(v)\,\mathrm{d}v <\infty$. Conversely,
a choice of $\varphi$ can lead to any exponent within $(0,1]$ if
$\gamma\le1/2$, or within $(0,1/\gamma- 1]$
if $\gamma>1/2$; see Figure~\ref{fig1}.

%
%
\begin{figure}

\includegraphics{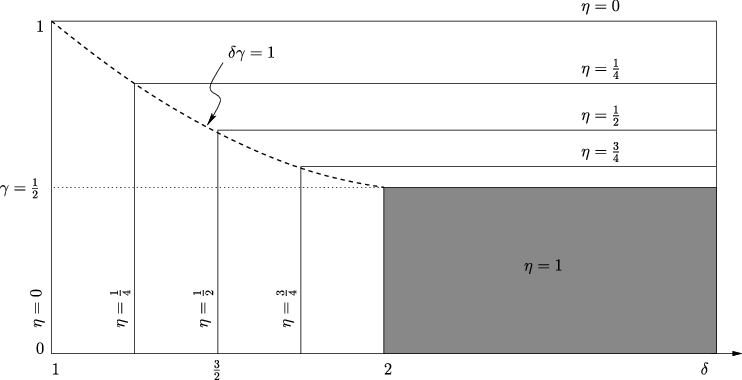}

\caption{Level sets for the length exponent~$\eta$ in the $(\delta,
\gamma)$ plane consist of
a rectangular block corresponding to the value $\eta=1$ and a family of
lines starting vertically at
the $\delta$-axis and turning horizontally upon hitting the graph given
by $\delta\gamma=1$.}\label{fig1}
\end{figure}

In Janssen, Pra\l{}at and
Wilson \cite{JPW} the empirical edge length distribution is
studied for the model defined in~\cite{Aiello}. This
is essentially the case of an affine function $f(k)=\gamma k+\beta$ and
a profile function $\varphi(x)=p\one\{ x<1/(2p)\}$, corresponding
roughly to the case $\delta=\infty$. They show that if $\gamma>\frac{1}2$
and $\frac{3\gamma+2}{4\gamma+2}<\alpha<1$, then
\[
\bigl| \bigl\{\mbox{edges of length longer than $t^{-\alpha}$} \bigr\} \bigr|
\sim C
t^{(2-\alpha)+(({1}/\gamma)(\alpha-1))}
\]
for an explicit constant $C>0$. Our result uses a different order of
limits, but leads to the same order of growth for the comparable
quantity $t \lambda[t^{1-\alpha},\infty)$. If $\gamma<\frac{1}2$ they
show that
the \emph{expected} number of edges of length longer than $t^{-\alpha
}$, for $0\leq\alpha<1$, grows of order $t^\alpha$, which is also
of the same order as $t \lambda[t^{1-\alpha},\infty)$. Note that the
general form of the profile functions allows for a genuinely richer
phenomenology in our case.

\section{Methods of proof}\label{sec4}

\subsection{The rescaled picture}\label{sec4.1}

First, it is convenient to describe more explicitly the randomness
involved in the ``construction rule,'' which determines
the presence or absence of each edge in the network. To this end,
denote by $\T_1\times(0,\infty)$ the set of \emph{potential vertices},
and by
\[
E\bigl(\T_1\times(0,\infty)\bigr):= \bigl\{ (\mathbf{y},\mathbf{x}),
\mathbf{y}, \mathbf{x}\in\T_1\times(0,\infty), \mathbf{y}\mbox{
younger than }\mathbf{x}\bigr\}
\]
the set of \emph{potential edges}. Introduce a family~$\V$ of
independent random variables, independent of $\X$, indexed by the set
of potential edges and uniformly distributed on $[0,1]$. We will denote
these variables by $\V_{\mathbf{x},\mathbf{y}}$ or $\V(\mathbf
{x},\mathbf{y})$. A realization of $\X_t$
and $\V$ defines a network $G^1(\X_t, \V)$, with vertex set $\X_t$,
obtained with the same construction as before, but with the
construction rule replaced by the rule that you connect $\mathbf{x}$
to $\mathbf{y}$ if
and only if
%
%
\begin{equation}
\label{attachmentrule} \V(\mathbf{x}, \mathbf{y}) \le\varphi\biggl
(\frac{s d(\mathbf
{x},\mathbf{y})} {f(Z_\mathbf{x}(s-))}
\biggr),
\end{equation}
where $s$ is the birth time of the younger vertex $\mathbf{y}$.
The growing networks $(G^1(\X_t, \V))_{t>0}$ and $(G_t)_{t>0}$ have the
same law and will be identified.
Moreover, the deterministic functional $G^1$ associates a graph
structure to any set of points
in $\T_1\times(0,\infty)$ and family of points in $[0,1]$ indexed by
$E(\T_1\times(0,\infty))$.

Second, we want to generalize the construction, replacing $\T_1$ by
$\T
_t=\R/(t\Z)$, the one-dimensional torus of length $t$. We permit the
case $t=\infty$, with the convention $\T_\infty=\R$. The definition of
the set of potential vertices $\T_t\times(0,\infty)$ and the set of
potential edges $E(\T_t\times(0,\infty))$ is straightforward.
We define the functional~$G^t$, for $t\in(0,\infty]$, in analogy to
the case $t=1$, by associating a graph structure to any set of points
in $\T_t\times(0,\infty)$, and any family of values in $[0,1]$ indexed
by $E(\T_t\times(0,\infty))$. In the construction,
rule~(\ref{attachmentrule}) is unchanged, but with the new
understanding that the distances are now those in $\T_t$.

For finite $t$, we introduce the rescaling mapping
\begin{eqnarray*}
h_t\dvt  \T_1\times(0,t] &\to& \T_t \times(0,1],
\\
(x,s) &\mapsto&(tx, s/t)
\end{eqnarray*}
which expands the space by a factor $t$, the time by a factor $1/t$.
The mapping $h_t$ operates on the set $\X$, but also on $\V$, with
\[
h_t(\V)_{h_t(\mathbf{x}),h_t(\mathbf{y})}: = \V_{\mathbf
{x},\mathbf{y}}.
\]
The operation of $h_t$ preserves the rule~(\ref{attachmentrule}), and
it is therefore simple to verify that we have
\[
G^t\bigl(h_t(\X_t), h_t(\V)
\bigr)= h_t\bigl(G^1(\X_t, \V)
\bigr)=h_t(G_t),
\]
that is, it is the same to construct the graph and then rescale the
picture, or to first rescale the picture, then construct the graph on
this rescaled picture. Observe also that $h_t(\X_t)$ is a Poisson
point process of
intensity $1$ on $\T_t \times(0,1]$, while $h_t(\V)$ is still an
independent family of i.i.d. uniform random variables on $[0,1]$,
indexed by $E(\T_t\times(0,1])$.

From now on, we denote by $\Y$ a Poisson point process with intensity
1 on $\R\times
(0,1]$, and $\V$ an independent family of i.i.d. uniform on $[0,1]$
random variables, indexed by $E(\R\times(0,1])$.
For finite $t>0$, identify $(-t/2, t/2]$ and $\T_t$, and write $\Y^t$
for the restriction of $\Y$ to $\T_t \times(0,1]$, and $\V^t$ for the
restriction of $\V$ to the indices in $E(\T_t \times(0,1])$. We write
$G^t(\Y, \V)$ for $G^t(\Y^t, \V^t)$, and observe that this graph has
the same law as $h_t(G_t)$. However, the process $ t\mapsto G^t(\Y, \V
)$ behaves very differently from the original
process $t\mapsto G_t$. Indeed, in the original process, the degree of
any fixed vertex grows like $t^{\gamma+ o(1)}$ (see Lemma~\ref
{polynomialgrowth}) and thus goes to $+\infty$. By contrast, for the
graphs $G^t(\X, \V)$, the following result establishes convergence to
the graph $G^\infty(\Y, \V)$ as defined in the preceding paragraph.

%
%
\begin{prop}\label{propositionlocalconvergence}
\textup{(i)} The graph $G^\infty(\Y, \V)$ defined above is almost surely
locally finite, in the sense that its vertices all have finite degrees.

\textup{(ii)} The graph $G^t(\Y, \V)$, almost surely, converges locally
to $G^\infty(\Y, \V)$, in the sense that for each $\mathbf{x}\in\Y
$, for
large $t$,
the neighbors of $\mathbf{x}$ in $G^t(\Y, \V)$ and in $G^\infty(\Y, \V)$ coincide.
\end{prop}

As a direct consequence we obtain the following corollary.\vadjust{\goodbreak}

%
%
\begin{cor}
Almost surely, for any $\mathbf{x}\in\Y$ and each $n\geq1$, the neighborhood
of vertex $\mathbf{x}$ in the
graphs $G^t(\Y, \V)$ and $G^\infty(\Y, \V)$ up to graph distance $n$
will coincide for large~$t$.
\end{cor}

The key to the understanding of the drastically different behavior of
the graph-valued process $t \mapsto G^t(\Y, \V)$
lies in the fact that a fixed vertex in this sequence of graphs has a
birth time which is comparable to the age of the network.
This age would be highly variable in time if mapped onto the original
graph, but is kept constant in the process $t \mapsto G^t(\Y, \V)$.

Regardless of the strength of Proposition~\ref
{propositionlocalconvergence}, it only states a local convergence
result and
is therefore insufficient for our purpose. Global results require the
introduction of a specific law of large numbers, which we state
and prove now.

\subsection{A general weak law of large numbers}\label{sec4.2}


For $x_0 \in\R$, we introduce the translation
\begin{eqnarray*}
\theta_{x_0}\dvt  \R\times(0,1] &
\to& \R\times(0,1],
\\
(x,s)& \mapsto&(x+x_0, s).
\end{eqnarray*}
The translation $\theta_{x_0}$ operates on $\R\times(0,1]$, and in a
canonical manner also on
the point sets in $\R\times(0,1]$, and on families indexed by $E(\R
\times(0,1])$. Consider a
functional $\xi_\infty$, which associates a nonnegative real number
$\xi
_\infty((x,s),\setZ, \W)$ to a point set $\setZ\subset\R\times(0,1]$
with a distinguished point $\mathbf{x}=(x,s) \in\setZ$, and a family
$\W$ of
numbers in $[0,1]$ indexed by $E(\R\times(0,1])$. The functional is
supposed to be translation invariant, in the sense that
\[
\xi_\infty(\mathbf{x},\setZ,\W)= \xi_\infty\bigl(\theta
_{x_0}(\mathbf{x}),\theta_{x_0}(\setZ),
\theta_{x_0}(\W)\bigr).
\]
Similarly, for each $t>0$, and $x_0 \in\T_t$, we introduce the translation
\begin{eqnarray*}
\theta^t_{x_0}\dvt \T_t \times(0,1] &\to& \T_t \times(0,1],
\\
(x,s)& \mapsto&(x+x_0, s),
\end{eqnarray*}
and we consider functionals $\xi_t$, which associate a nonnegative real
number $\xi_t((x,s),\setZ, \W)$
to a point set $\setZ\subset\T_t \times(0,1]$ with a distinguished
point $(x,s) \in\setZ$
and a family $\W$ of numbers in $[0,1]$ indexed by $E(\T_t \times
(0,1])$. The functionals $\xi_t$ are
supposed to be invariant under the translations $\theta^t_{x_0}$.

Finally, for the sake of simplifying notation, we will also write $\xi
_\infty(\mathbf{x},\setZ, \W)$ for $\xi_\infty(\mathbf{x},\setZ
\cup\{\mathbf{x}\}, \W)$ when
the set $\setZ$ does not contain $\mathbf{x}$, and similarly $\xi
_t(\mathbf{x},\setZ, \W
)$ for $\xi_t(\mathbf{x},\setZ\cup\{\mathbf{x}\}, \W)$.
We also write
\begin{eqnarray*}
\xi_\infty(\setZ, \W) &:=& \int_0^1
\xi_\infty\bigl((0,s), \setZ, \W\bigr)\,\mathrm{d}s,
\\
\xi_t(\setZ, \W) &:=& \int_0^1
\xi_t\bigl((0,s), \setZ, \W\bigr) \,\mathrm{d}s.
\end{eqnarray*}
Recall the notation of the Poisson point process $\Y$ and of the
family of random
variables~$\V$, as well as their restrictions $\Y^t$ and $\V^t$.
In the following theorem, $U$ denotes a random variable, uniform on
$(0,1]$, and independent of the point process~$\Y$ and of~$\V$.

%
%
\begin{theorem}[(Weak law of large numbers)]\label{theoLLN}
Suppose that the following two conditions hold:
\begin{longlist}[(A)]
\item[(A)] as $t\to\infty$, the random variable $\xi_t((0,U), \Y
^t, \V
^t)$ converges in probability to the
random variable~$\xi_\infty((0,U), \Y, \V)$;
\item[(B)] for some $p>1$ we have the uniform moment condition
\[
\sup_{t>0} \mathbb E\bigl[\xi_t\bigl((0,U),
\Y^t, \V^t\bigr)^p\bigr]<\infty.
\]
\end{longlist}
Then, as $t\to\infty$, we have the following convergence in the $L^1$-sense:
%
%
\begin{equation}
\label{LLN} \frac{1} t \sum_{\mathbf{x}\in\Y^t}
\xi_t\bigl(\mathbf{x}, \Y^t, \V^t\bigr)
\longrightarrow\mathbb E\bigl[\xi_\infty\bigl((0,U),\Y, \V\bigr)\bigr]=
\mathbb E\bigl[\xi_\infty(\Y, \V)\bigr].
\end{equation}
\end{theorem}

%
%
\begin{rem} \label{remLLN}
\textup{(i)} Theorem~\ref{theoLLN} is an adaptation of Theorem~2.1 of
Penrose and Yukich~\cite{YukichPenrose} to our purpose. Their result
also includes a de-Poissonisation, but this is incompatible with our
set-up because of the explicit time dependence of the attachment
probabilities.

\textup{(ii)} Suppose now that only condition \textup{(A)} is satisfied. On the
one hand, the proof still works
if the family $(\xi_t((0,U), \Y^t, \V^t))_{t>0}$
is uniformly integrable. On the other hand, if $ \mathbb E[\xi_\infty
(\Y, \V
)]=\infty$, then, by applying the theorem to the bounded functional
$\xi^k_t(\mathbf{x}, \setZ, \W):=\xi_t(\mathbf{x}, \setZ, \W)
\wedge k$
and letting $k$ go to $\infty$, we get the convergence in probability of
\[
\frac{1} t \sum_{\mathbf{x}\in\Y^t}\xi_t
\bigl(\mathbf{x}, \Y^t, \V^t\bigr)
\]
to $+\infty$. The only case when the theorem does not yield any
convergence result is when
$ \mathbb E[\xi_\infty( \Y, \V)]$ is finite, but the family $(\xi
_t((0,U), \Y
^t, \V^t))_{t>0}$ fails to be
uniformly integrable.
\end{rem}

\begin{pf}
As in Theorem~2.1 in~\cite{YukichPenrose} the proof relies on a first
moment calculation, and then a second moment calculation which is
performed under
a stronger uniform moment condition, and finally a step to allow the
removal of this extra condition.

\emph{First moment}: We compute, by Campbell's formula,
\begin{eqnarray*}
\mathbb E \biggl[\frac{1} t \sum_{\mathbf{x}\in\Y^t}
\xi_t\bigl(\mathbf{x}, \Y^t, \V^t\bigr)
\biggr] & =& \int_{\T_t \times(0,1]} \frac{\mathrm{d}x \,\mathrm{d}s}t
\mathbb E \bigl[
\xi_t\bigl((x,s), \Y^t, \V^t \bigr) \bigr]
\\
&=& \int_0^1 \mathrm{d}s\,\mathbb E \bigl[
\xi_t\bigl((0,s), \Y^t, \V^t\bigr) \bigr]
\\
&=&
\mathbb E \bigl[\xi_t\bigl((0,U), \Y^t, \V^t
\bigr) \bigr].
\end{eqnarray*}
Note that in all these expressions but the first one, a point is added
to $\Y^t$. The second equality follows from the spatial invariance
by the translation $\theta_{-x}^t$, both of the functional $\xi_t$ and
of the law of $(\Y^t, \V^t)$.
Now condition~\textup{(A)} states that the variables $\xi_t((0,U), \Y^t, \V^t)$
converge in probability to
$\xi_\infty((0,U),\Y, \V)$. Condition~\textup{(B)} ensures that they are
uniformly integrable. Therefore we have convergence of the expectations
$\mathbb E[\xi_t((0,U), \Y^t, \V^t)]$ to $\mathbb E[\xi_\infty
((0,U), \Y, \V)]$
, and this expectation is finite.

\emph{Second moment:} We work here under the stronger assumption that
the uniform moment condition holds for some $p>2$. Similarly as in the
case of
the first moment, we get
\begin{eqnarray*}
&& \mathbb E \biggl[  \biggl(\frac{1} t \sum_{\mathbf{x}\in\Y^t}
\xi_t \bigl(\mathbf{x}, \Y^t, \V^t \bigr)
\biggr)^2 \biggr]
\\
&&\qquad = \mathbb E \biggl[\frac{1} {t^2} \sum_{\mathbf{x}\in\Y^t}
\xi_t \bigl(\mathbf{x},\Y^t, \V^t
\bigr)^2 \biggr] + \mathbb E \biggl[\frac{1} {t^2} \mathop{\sum
_{\mathbf{x},
\mathbf{x}' \in\Y^t}}_{\mathbf{x}\ne\mathbf{x}'} \xi_t \bigl(
\mathbf{x}, \Y^t, \V^t \bigr) \xi_t \bigl(
\mathbf{x}', \Y^t, \V^t \bigr) \biggr]
\\
&&\qquad = \frac{1} t \mathbb E \bigl[ \xi_t \bigl((X_1,U_1),
\Y^t, \V^t \bigr)^2 \bigr]
\\
&&\quad\qquad{}  + \mathbb E \bigl[\xi_t \bigl((tX_1,U_1),
\Y^t \cup\bigl\{ 
(tX_2,U_2)\bigr
\}, \V^t \bigr)
\\
&&\hspace*{54pt}{}\times  \xi_t \bigl((tX_2,U_2),
\Y^t \cup\bigl\{(tX_1,U_1) \bigr\},
\V^t \bigr) \bigr],
\end{eqnarray*}
with $X_1$ and $X_2$ uniform in $\T_1$, $U_1$ and $U_2$ uniform in
$(0,1)$, and $\Y$, $X_1$, $U_1$, $X_2$, $U_2$ independent.
The first term goes to zero, thanks to the uniform moment condition
with $p>2$ ($p=2$ would be enough).

Now, the second term is the expectation of the following product of
random variables:
%
%
\begin{eqnarray}\label{product}
&& \xi_t\bigl((0,U_1),
\theta^t_{-tX_1}\bigl(\Y^t\bigr) \cup\bigl\{
\bigl(t(X_2-X_1),U_2\bigr)\bigr\}, \theta
^t_{-tX_1}\bigl(\V^t\bigr)\bigr)
\nonumber\\[-8pt]\\[-8pt]
&&\quad{} \times\xi_t\bigl((0,U_2),
\theta^t_{-tX_2}\bigl(\Y^t\bigr) \cup\bigl\{
\bigl(t(X_1-X_2),U_1\bigr)\bigr\},
\theta^t_{-tX_2}\bigl(\V^t\bigr)\bigr),\nonumber
\end{eqnarray}
whose behavior we have to understand. We first concentrate on the first term.
Write
\[
\widetilde\Y^t: =\theta^t_{-tX_1}\bigl(
\Y^t\bigr) \cup\bigl\{\bigl(t(X_2-X_1),U_2
\bigr)\bigr\},\qquad \widetilde\V^t: =\theta^t_{-tX_1}
\bigl(\V^t\bigr).
\]
We introduce three events, $E_t:=\{t d(X_1,X_2)>\sqrt t\}$,
$F_t:= \{t d(X_1, 1/2)>\sqrt t /2\}$ and
$G_t$ the event that the Poisson point process $\Y^t$ has at least one
point in $\{
(x,s) \dvtx  d(x,0)> \sqrt t \}$. These are all
\emph{asymptotically almost sure} (\emph{a.a.s.}), in the sense that their
probability goes to one when $t\to\infty$.
We make two important observations:
\begin{itemize}
\item On the event $E_t\cap F_t$, the restrictions to $\T_{\sqrt
t}\times(0,1]$ of the sets $\widetilde\Y^t$ and $\theta_{-tX_1}(\Y
)$ coincide.
Similarly, the restrictions to $\T_{\sqrt t}\times(0,1]$ of the
families $\widetilde\V^t$ and $\theta_{-tX_1}(\V)$ also coincide.\vspace*{1pt}
\item The law of $(\widetilde\Y^t, \widetilde\V^t)$ knowing $E_t$
equals the law of
$(\Y^t, \V^t)$ knowing $G_t$.
\end{itemize}

These observations allow the following calculation, with $\eta$ some
positive real number.
Note that we will apply now (and until the end of this proof) the
functional $\xi_{\sqrt t}$ to point sets on $\R\times(0,1]$ or $\T_t
\times(0,1]$ (and families indexed by $E(\R\times(0,1])$ or $E(\T_t
\times(0,1])$). This is only to lighten the notation a bit. It should
always be understood that the functional is applied to the restrictions
on $\T_{\sqrt t} \times(0,1]$.
\begin{eqnarray*}
&& \limsup_{t\to\infty}  \P\bigl\{\bigl|\xi_t
\bigl((0,U_1), \widetilde\Y^t, \widetilde\V^t
\bigr) - \xi_{\sqrt t} \bigl((0,U_1), \theta_{-tX_1}(
\Y), \theta_{-tX_1}(\V)\bigr) \bigr| >\eta\bigr\}
\\
&&\qquad = \limsup_{t\to\infty} \mathbb E \bigl[\un\bigl\{\bigl|\xi_t
\bigl((0,U_1), \widetilde\Y^t, \widetilde\V^t
\bigr)
\\
&&\hspace*{88pt}{}- \xi_{\sqrt t} \bigl((0,U_1), \theta_{-tX_1}(
\Y), \theta_{-tX_1}(\V)\bigr) \bigr| >\eta\bigr\} \un_{F_t} |
E_t \bigr]
\\
&&\qquad = \limsup_{t\to\infty} \mathbb E \bigl[\un\bigl\{\bigl|\xi_t
\bigl((0,U_1), \widetilde\Y^t, \widetilde\V^t
\bigr) - \xi_{\sqrt t} \bigl((0,U_1), \widetilde
\Y^t, \widetilde\V^t\bigr) \bigr| >\eta\bigr\}
\un_{F_t} | E_t \bigr]
\\
&&\qquad = \limsup_{t\to\infty} \mathbb E \bigl[\un\bigl\{\bigl|\xi_t
\bigl((0,U_1), \Y^t, \V^t\bigr) - \xi
_{\sqrt t} \bigl((0,U_1), \Y^t, \V^t
\bigr) \bigr| >\eta\bigr\} \un_{F_t} | G_t \bigr]
\\
&&\qquad = \limsup_{t\to\infty} \P\bigl\{\bigl|\xi_t
\bigl((0,U_1), \Y^t, \V^t\bigr) - \xi
_{\sqrt t} \bigl((0,U_1), \Y^t, \V^t
\bigr) \bigr| >\eta\bigr\} = 0.
\end{eqnarray*}
The last equality uses condition \textup{(A)}.
Hence, the variable
\[
\xi_t\bigl((0,U_1), \widetilde\Y^t,
\widetilde\V^t\bigr) - \xi_{\sqrt t} \bigl((0,U_1),
\theta_{-tX_1}(\Y), \theta_{-tX_1}(\V)\bigr)
\]
converges in probability to zero. Similarly, one can see that the variable
\begin{eqnarray*}
&& \xi_t \bigl((0,U_2), \theta^t_{-tX_2}
\bigl(\Y^t\bigr) \cup\bigl\{\bigl(t(X_1-X_2),U_1
\bigr)\bigr\}, \theta^t_{-tX_2}\bigl(\V^t\bigr)
\bigr)
\\
&&\quad {}  - \xi_{\sqrt t} \bigl((0,U_2), \theta_{-tX_2}(
\Y), \theta^t_{-tX_2}(\V) \bigr)
\end{eqnarray*}
converges in probability to zero. Next, observe that the two variables
\[
\xi_{\sqrt t}\bigl((0,U_1), \theta_{-tX_1}(\Y),
\theta_{-tX_1}(\V)\bigr)
\]
and
\[
\xi_{\sqrt t}\bigl((0,U_2), \theta_{-tX_2}(\Y),
\theta_{-tX_2}(\V)\bigr)
\]
are independent conditionally on the event $E_t$. Moreover, observe
that the law of each one converges to that of $\xi_\infty((0,U_1), \Y,
\V)$, thanks to condition \textup{(A)} again.
Gathering the results, we get that the product in~(\ref{product})
converges in law to the product of two independent copies of $\xi
_\infty
((0,U_1), \Y, \V)$.

Finally, use Cauchy--Schwarz to get a uniform moment condition for this
product for $\frac{p} 2>1$. Hence the expectation of the product goes to
$\mathbb E[\xi_\infty((0,U_1),\break \Y, \V)] ^2$. Therefore we get~(\ref{LLN}),
with convergence even in $L^2$.

\emph{Relaxing the moment condition}: We finally work under the
assumptions of the theorem, that is, the uniform moment condition is
satisfied only for some $p>1$. Introduce the bounded functional
\[
\xi^k_t(\mathbf{x}, \setZ, \W):=\xi_t(
\mathbf{x}, \setZ, \W) \wedge k.
\]
This functional clearly satisfies condition \textup{(A)} and the uniform moment
condition for any $p$, in particular for some $p>2$. Therefore, we get
the convergence of
\[
\frac{1} t \sum_{\mathbf{x}\in\Y^t} \xi_t^k
\bigl(\mathbf{x}, \Y^t, \V^t\bigr)
\]
to $\mathbb E[\xi_\infty((0,U_1),\Y, \V)\wedge k ]$, in $L^2$, and
thus in $L^1$.
Now note that
\begin{eqnarray*}
&& \mathbb E \biggl[\frac{1}t \sum
_{(x,s) \in\Y^t}  \bigl(\xi_t\bigl((x,s), \Y^t,
\V^t\bigr) -\xi_t^k\bigl((x,s),
\Y^t, \V^t\bigr)\bigr) \biggr]
\\
&&\qquad  = \mathbb E \bigl[\xi_t\bigl((0,U), \Y^t,
\V^t\bigr)- \xi_t^k\bigl((0,U),
\Y^t, \V^t\bigr) \bigr],
\end{eqnarray*}
which is nonnegative and goes uniformly to zero for $k\to\infty$, as
the variables $\xi_t((0,U), \Y^t, \V^t)$ are uniformly integrable,
by the uniform moment condition. It follows that
\[
\frac{1} t \sum_{\mathbf{x}\in\Y^t} \xi_t
\bigl(\mathbf{x}, \Y^t, \V^t\bigr)
\]
converges in $L^1$ to the limit of $\mathbb E[\xi_\infty((0,U),\Y,
\V)\wedge k
]$, that is $\mathbb E[\xi_\infty((0,U),\break\Y, \V)]$.
\end{pf}

\subsection{A bound on the indegree and on the linking probability}\label{sec4.3}

As we consider various graphs on various spaces, we need to introduce
more flexible notation for the degrees.
If $G$ is a graph with vertices in $\T_t \times(0,\infty)$ we write
$\mathbf{x}\leftrightarrow\mathbf{y}$ to indicate that there is an
edge between
the vertices $\mathbf{x}$ and $\mathbf{y}$.
If $\mathbf{x}_0=(x_0,s_0)$ is in $G$, then, for any $s\ge s_0$, we define
\[
Z_{\mathbf{x}_0}(s, G) = \bigl| \bigl\{(x,r) \in G \dvtx
(x,r)\leftrightarrow
(x_0,s_0), s_0 <r \le s \bigr\} \bigr|,
\]
the indegree of $\mathbf{x}_0$ in $G$ ``at time $s$'' and
\[
Y_{\mathbf{x}_0}(G) = \bigl| \bigl\{(x,r) \in G \dvtx (x,r)\leftrightarrow
(x_0,s_0), r <s_0 \bigr\} \bigr|,
\]
its outdegree. For $t \in(0,\infty]$ and for $0<s_0\le s \le1$, we write
\[
Z^t_{\mathbf{x}_0}(s) = Z_{\mathbf{x}_0}\bigl(s,G^t
\bigl(\Y\cup\{\mathbf{x}_0\}, \V\bigr)\bigr)\quad\mbox{and}\quad
Y^t_{\mathbf{x}_0}=Y_{\mathbf
{x}_0}\bigl( G^t\bigl(\Y
\cup\{\mathbf{x}_0\}, \V\bigr)\bigr).
\]
For fixed $t$ and $\mathbf{x}_0$, call $(Z^t_{\mathbf
{x}_0}(s))_{s_0\le s \le1}$ the
\emph{indegree process}. In this part only, we extend the Poisson
point process $\X$ on
the whole $\R\times(0,\infty)$, and allow any $0<s_0\le s$ in the
definition of $Z^t_{\mathbf{x}_0}(s)$. For $\mathbf{x}_0 \in\T_1
\times(0,\infty)$,
the process $(Z^1_{\mathbf{x}_0}(s))_{s\ge s_0}$ has the same law as the
process $(Z_{\mathbf{x}_0}(s))_{s\ge s_0}$ introduced earlier in
Section~\ref{sec2}, so
that the results of this part apply simultaneously for the rescaled
graphs and for the unrescaled ones.
Now, observe that the law of the indegree process does not depend on
the spatial position $x_0 \in\T_t$. Therefore, we
simply write $Z^t_{s_0}(s)$ for $Z^t_{(0,s_0)}(s)$ and $Y^t_{s_0}$ for
$Y^t_{(0,s_0)}$. If $\mathbf{x}$ and $\mathbf{y}$ are two vertices in
$\Y$, we write
$\mathbf{x}\underset t {\leftrightarrow} \mathbf{y}$ for the event
that $\mathbf{x}$ and $\mathbf{y}$
are linked in
$G^t(\Y, \V)$.

\begin{lemma}\label{polynomialgrowth}
For all $t>0$ and $\mathbf{x}_0\in\T_t$, we have almost surely
\[
\ln Z^t_{\mathbf{x}_0}(s) \sim\gamma\ln s\qquad\mbox{as } s \to
\infty.
\]
\end{lemma}

This lemma confirms that the degree of a fixed vertex in the unrescaled
graphs grows polynomially of order~$\gamma>0$,
and in particular that it explodes.
Before proving it 
we give a bound on the probability that a vertex reaches an
exceptionally high degree,
allowing it to be connected to an exceptionally distant vertex.
Exponential bounds, uniform in $t$, are provided in the following lemma
and its corollary. For the sake of simplicity, they are only stated in
the case of a linear function $f$. We refer to Remark~\ref{genrem} for the
general case.

%
%
\begin{lemma} \label{exponentialbound}
Suppose $f(k)=\gamma k + \beta$, with $\gamma\in(0,1)$ and $\beta>0$.
Let $p= \lceil\frac{\beta} \gamma-1\rceil$, so that $f(k)\le
\gamma
(k+p+1)$. For any $t \in(0,\infty]$, any $s_0<s\le1$ and any $k\ge0$,
the following inequality holds:
%
%
\begin{equation}
\label{boundindegree} \P\bigl\{ Z^t_{s_0}(s)\ge k \bigr\} \le
e^{{p}/ 4} \exp\biggl(-\frac{k} 8 \biggl(\frac{s_0} s
\biggr)^{\gamma} \biggr).
\end{equation}
\end{lemma}
%

%
%
\begin{cor} \label{linkingbound}
Under the assumptions of Lemma~\ref{exponentialbound}, define the
inverse of the profile function~$\varphi$ by
\[
\varphi^{-1}(u):= \inf\bigl\{x>0 \dvtx \varphi(x)< u\bigr\}.
\]
Then there is a constant $c$ depending only on $f$, such that for any
$t \in(0,\infty]$ and any $(x,s)\in\R\times(0,1]$, we have
%
%
\begin{eqnarray}\label{linkingboundequation}
&& \P\bigl\{ (0,1) \underset t
{\leftrightarrow} (x,s) | (0,1) \in
\Y, (x,s) \in\Y, \V_{(0,1),(x,s)}=u \bigr\}
\nonumber\\[-8pt]\\[-8pt]
&&\qquad \le c \exp\biggl( -
\frac{|x| s^\gamma} {8 \gamma\varphi^{-1}(u)} \biggr).\nonumber
\end{eqnarray}
\end{cor}

%
%
\begin{rem}\label{genrem}
In the nonlinear case, we can first bound $f$ from above by a
linear function, then, by an easy stochastic domination argument, get
the inequalities of the lemma and its corollary with the linear bound
instead of $f$. We get almost equally good bounds. More precisely, for
any $\gamma'>\gamma$, we can find $\beta'>0$ such that $f(k) \le
\gamma
' k + \beta'$ for any natural number $k$, and we thus get bounds for
any exponent $\gamma'>\gamma$.
\end{rem}

A first corollary of Lemma~\ref{exponentialbound} is that the indegree
$Z^t_{\mathbf{x}_0}(s)$ is always almost surely finite,
even when $t=\infty$. The same holds for the outdegree; see
Proposition~\ref{limitingoutdegree} below.

At this stage, let us discuss the important \emph{monotonicity
property}. If we fix $s_0$ and $s$ and let $t$ grow to $+\infty$, then
$Z^t_{s_0}(s)$ will grow and converge to $Z^\infty_{s_0}(s)$. Moreover,
if we change the position of the vertex to be nonzero, we do not
change the law of its indegree and therefore its indegree will still be
\emph{stochastically increasing} in~$t$ and \emph{stochastically
dominated} by $Z^\infty_{s_0}(s)$.
By contrast, no such property holds for the outdegree. Indeed,
increasing~$t$ may increase the distance of two vertices near opposite
ends of
the boundary of $[-\frac{t}2, \frac{t}2]$, thus
decreasing the indegree of the younger vertex which, in turn, might
destroy further links, eventually reducing the outdegree of the vertex
at the origin.

\begin{pf*}{Proof of Lemma~\ref{polynomialgrowth}}
We fix $s_0>0$ and start with the case $t=\infty$.
The indegree process $(Z^\infty_{s_0}(s))_{s>s_0}$ is an
time-inhomogeneous\vspace*{1pt} pure birth process, starting from $Z^\infty
_{s_0}(s_0)=0$, and for which, at time $s$, the transition density from
state $k$ to state $k+1$ is
$f(k)/s$. Indeed, given $Z_{s_0}(s)=k$, we have $Z^\infty
_{s_0}(s+\mathrm{d}s)\ge k+1$ if and only if the set
\[
\biggl\{(y,u)\in\Y\dvtx  u \in(s, s+\mathrm{d}s], \V\bigl
((0,1),(y,u)\bigr) \le
\varphi\biggl( \frac{u d(y,0)} {f(k)} \biggr) \biggr\}
\]
is nonempty, which due to the normalization of~$\varphi$ happens with
probability $\frac{f(k)}s \,\mathrm{d}s + o(\mathrm{d}s)$.
We introduce a logarithmic change of time and write
\[
\widetilde Z(u): = Z^\infty_{s_0}\bigl(s_0
e^u\bigr).
\]
Then\vspace*{1pt} the
process $\widetilde Z$ is a time-homogeneous pure birth process, with jumping
intensity from state
$k$ to state $k+1$ equal to $f(k)$. Write $T_k:= \inf\{u \dvtx
\widetilde
Z(u)\ge k\}$ for the first time when
this process hits state~$k$, which is finite as $f$ is nondecreasing.
Then $(T_{i+1}- T_i)_{i\ge0}$ are independent, and $T_{i+1}- T_i$ is
exponential with parameter $f(i)$. The process
\[
M_k:= T_k - \sum_{i=0}^{k-1}
\frac{1} {f(i)}
\]
is a martingale, which is bounded in $L^2$ and thus convergent. Hence,
we have
$T_k \sim\frac{1} \gamma\ln k$, and further
\[
\ln\widetilde Z(u) \sim\gamma u\quad\mbox{and}\quad \ln Z^\infty_{s_0}(s)
\sim\gamma\ln s.
\]
For the case of a finite $t$, we first get, from the monotonicity
property, the upper bound
\[
\limsup_{s\to\infty} \frac{\ln Z^t_{s_0}(s)} {\ln s} \le\gamma.
\]
In particular, a.s., we have $Z^t_{s_0}(s)\le s^{(1+\gamma)/ 2}$ for
$s$ large enough.
But the process $(Z^t_{s_0}(s))_{s>s_0}$ is a time-inhomogeneous pure
birth process with transition density from state $k$ to state $k+1$
\[
2f(k) \int_0^{({st})/({f(k)})} \varphi(y) \,\mathrm{d}y,
\]
which is equivalent to $f(k)$ when $t\uparrow\infty$, uniformly for all
$s$ and $k\le s^{(1+\gamma)/2}$.
The same arguments as in the case $t=\infty$
then yield the lower bound, showing that we still have $\ln
Z^t_{s_0}(s) \sim\gamma\ln s$.
\end{pf*}

\begin{pf*}{Proof of Lemma~\ref{exponentialbound}}
By the monotonicity argument we can suppose $t=\infty$ and, as before,
we study the chain $\widetilde Z$ and its hitting times $T_k$. The parameter
of the exponential variable $T_{i+1}- T_i$ is $f(i)$, which is less
than or equal to $(p+i+1) \gamma$. It follows that $T_k$
(!!CHANGE!!) dominates stochastically
a sum of independent exponential random variables with parameters
$(p+1)\gamma$, $(p+2) \gamma, \ldots, (p+k) \gamma$, respectively.

Let $(\tilde\tau_i)_{1\le i \le k+p}$ be a family of i.i.d. random
variables, each following an exponential law with the same parameter
$\gamma$. Let $(\tilde\tau_{i_1}, \tilde\tau_{i_2}, \ldots, \tilde\tau
_{i_{k+p}})$ denote their decreasing rearrangement, and $\tilde\tau
_{i_{k+p+1}}=0$. For $1\le j\le k+p$, let $\tau_j= \tilde\tau
_{i_j}- \tilde
\tau_{i_{j+1}}$. Then the family $(\tau_j)_{1\le j \le k+p}$ is
independent, and $\tau_j$ is an exponential variable with parameter $j
\gamma$.
Observe also that
\[
\tau_{p+1}+\cdots+\tau_{p+k}=\tilde\tau_{i_{p+1}}.
\]
Hence,
\[
\P\bigl\{Z^\infty_{s_0}\bigl( s_0
e^u\bigr)\ge k \bigr\} \le\P\{ \tilde\tau_{i_{p+1}} \le u
\}.
\]
Now write
\[
\{\tilde\tau_{i_{p+1}} \le u \} = \Biggl\{\sum
_{j=1}^{k+p} \un\{\tilde\tau_j >u\}
\le p \Biggr\}.
\]
The sum of indicators follows a binomial law of parameters $k+p$ and
$\exp(-\gamma u)$. Recall the concentration inequality for
binomial random\vadjust{\goodbreak} variables~$X$,
\[
\P\bigl\{X\le\mathbb E[X] - \lambda\bigr\} \le\exp\biggl(-\frac
{\lambda^2} {2
\mathbb E[X]}
\biggr).
\]
We apply this with $\lambda= \frac{1} 2 (k+p) \exp(-\gamma u)$ and get
\begin{eqnarray*}
\P\Biggl\{\sum_{j=1}^{k+p} \un
\{\tilde\tau_j >u\} \le p \Biggr\} & \le&\exp\biggl(-
\frac{k} 8 e^{-\gamma u} \biggr) \un\bigl\{2p\le k e^{-\gamma
u}
\bigr\} + \un\bigl\{2p> k e^{-\gamma u}\bigr\}
\\
& \le&\exp(p/4)\exp\biggl(- \frac{k} 8 e^{-\gamma u} \biggr).
\end{eqnarray*}
Finally, gathering the results, and taking $u=\ln s - \ln s_0$ gives,
for any $k\ge0$,
\[
\P\bigl\{Z^{\infty}_{s_0}( s) \ge k \bigr\} \le\exp(p/4)\exp
\biggl(- \frac{k} 8 \biggl(\frac{s_0} s \biggr)^\gamma
\biggr),
\]
as required.
\end{pf*}

\begin{pf*}{Proof of Corollary~\ref{linkingbound}}
The event $(0,1) \underset{t}{\leftrightarrow} \mathbf{x}$
coincides with the event that the indegree of
vertex $\mathbf{x}$ at time one
is large enough to ensure that the linking condition is satisfied. This
indegree has the same law as $Z^t_s(1)$ and is
independent of $\V((0,1),\mathbf{x})$. We thus get
\begin{eqnarray*}
&& \P\bigl\{(0,1) \underset{t} {\leftrightarrow} \mathbf{x} | (0,1) \in\Y,
\mathbf{x}\in\Y, \V\bigl((0,1),\mathbf{x}\bigr)=u \bigr\}
\\
&&\qquad  \le \P
\biggl\{\varphi
\biggl(\frac{|x|} {f(Z^t_s(1))} \biggr) \ge u \biggr\}
\\
&&\qquad \le\P\biggl\{ Z^t_s(1) \ge f^{-1} \biggl(
\frac{|x|} {\varphi
^{-1}(u)} \biggr) \biggr\}
\\
&&\qquad  \le e^{p/4} \exp\biggl( -
\frac{s^\gamma} 8 \biggl(\frac{|x|} {
\gamma\varphi^{-1}(u)} - \frac\beta\gamma\biggr) \biggr)
\\
&&\qquad \le e^{(p/4) + (\beta/(8 \gamma))} \exp\biggl( -\frac{|x|
s^\gamma} {8 \gamma\varphi^{-1}(u)} \biggr),
\end{eqnarray*}
yielding~(\ref{linkingboundequation}) with the explicit constant $c=
e^{(p/4) + (\beta/(8 \gamma))} $.
\end{pf*}

\section{Specific proofs of the main results}\label{sec5}

All the proofs of this section rely on the application of Theorem~\ref
{theoLLN} to appropriate functionals.
The functionals we use are only defined and used within each
subsection. That is, the same notation in different subsections
indicates different functionals.

\subsection{Empirical indegree distribution}\label{sec5.1}\label{subsectionindegree}
The following lemma provides the expected indegree of a vertex in the
infinite graph with age uniform on $(0,1]$.\vadjust{\goodbreak}

%
%
\begin{lemma}\label{limitingindegreemeasure}
Let $U$ be uniformly distributed in $(0,1]$ and independent of the
point process $\Y$. Then, for any $k\ge0$, we have
\[
\P\bigl\{Z^\infty_U( 1)=k\bigr\}= \mu(k),
\]
where $\mu$ is the probability measure defined by
%
%
\begin{equation}
\label{degreedistribution} \mu(k)= \frac{1} {1+f(k)} \prod_{l=0}^{k-1}
\frac{f(l)} {1+f(l)}.
\end{equation}
\end{lemma}

\begin{pf}
Recall that the process $(Z^\infty_{s_0}( s_0 e^u))_{0\le u\le\ln(1/s_0)}$
is a time-homoge\-neous pure birth process with transition intensity from
state $k$ to state $k+1$ equal to $f(k)$.
Consider also the Markov chain $(\widehat Z_u)_{0\le u\le\ln(1/s_0)}$
with values in $[s_0,1]\times\N\cup\{0\}$ started in $\widehat Z_0=(s_0,0)$, such that
at time $u$ the jumping intensity from state $(s,k)$ to state $(s,k+1)$
equals $f(k)$, and from state $(s,k)$ to state $(s_0 e^u,0)$ equals one.

The following facts are easy to check:
\begin{longlist}[(3)]
\item[(1)] The\vspace*{2pt} first coordinate $\widehat Z_{\ln(1/s_0)}^1$ of the chain $(\widehat Z_u)_{0\le u\le\ln(1/s_0)}$ at time $\ln(1/s_0)$
is equal to $s_0$ with probability $s_0$ and otherwise uniformly
distributed on the interval $[s_0,1]$.
\item[(2)] Conditionally\vspace*{1pt} on $\widehat Z_{\ln(1/s_0)}^1=s_1$, the second
coordinate $\widehat Z_{\ln(1/s_0)}^2$ has the same law as the random
variable~$Z^\infty_{s_1}(1)$.
\item[(3)] The second coordinate $(\widehat Z^2_u)_{0\le u\le\ln(1/s_0)}$ is a
time-homogeneous Markov chain, jumping from $k$ to $k+1$ with intensity
$f(k)$, and from $k$ to zero with intensity one.
\end{longlist}
The Markov chain stated in the third point was already encountered
in~\cite{DMdegrees}. It is recurrent and its law converges to its
invariant measure, which is\vspace*{1pt} precisely $\mu$.
From the first two points, we deduce that the law of $\widehat Z_{\ln
1/s_0}^2$ conditional on
$\widehat Z_{\ln1/s_0}^1\neq s_0$ is the same as the law of $Z^\infty
_U(t)$, where $U$ is uniform on $[s_0, 1]$.
Now, letting $s_0$ go to zero gives the result.
\end{pf}

\begin{pf*}{Proof of Theorem~\ref{theoconvergenceindegree}}
Let $g$ be a nondecreasing functional satisfying $\langle\mu,
g^p\rangle<\infty$ for some $p>1$.
We will apply Theorem~\ref{theoLLN} with the functionals $\xi
_t(\mathbf{x},
\setZ, \W):= g (Z_\mathbf{x}(1, G^t(\setZ\cup\{\mathbf{x}\},
\W)) )$, $t\in
(0,\infty]$, so that for $\mathbf{x}\in\Y^t$, we have $\xi
_t(\mathbf{x}, \Y, \V
)=g(Z^t_\mathbf{x}(1))$.

First, observe that the expectation of $\xi_\infty(\Y, \V)$ is
$\langle
\mu,g\rangle$.
Second, observe the following two simple consequences of the
monotonicity property. The process $(Z^t_U(1))_{t>0}$ is nondecreasing
and converges
almost surely to $Z^\infty_U(1)$, which is finite almost surely.
Moreover, the following uniform moment condition is satisfied:
\[
\sup_{t>0} \mathbb E\bigl[\xi_t\bigl((0,U), \Y, \V
\bigr)^p\bigr] \le\mathbb E\bigl[\xi_\infty\bigl((0,U), \Y,
\V\bigr)^p\bigr] = \bigl\langle\mu, g^p\bigr\rangle<
\infty.
\]
Hence, Theorem~\ref{theoLLN} ensures the convergence
\[
\frac{1} t \sum_{\mathbf{x}\in\Y^t} g
\bigl(Z_\mathbf{x}^t(1) \bigr) \longrightarrow\langle\mu,g
\rangle,
\]
in $L^1$ and thus in probability. Combining this with the well-known convergence
${|\Y^t|}/t \to1$ gives the convergence in probability
\[
\frac{1} {|\Y^t|} \sum_{\mathbf{x}\in\Y^t} g
\bigl(Z_\mathbf{x}^t(1) \bigr) \longrightarrow\langle\mu,g
\rangle,
\]
and thus proves Theorem~\ref{theoconvergenceindegree}.
\end{pf*}

We close this subsection with a lemma which implies Proposition~\ref{propositionlocalconvergence}(i).

%
%
\begin{lemma}\label{lemmalocalconvergenceindegree}
Almost surely, for any $\mathbf{x}=(x,s)\in\Y$, the incoming edges
of $\mathbf{x}$ in
$G^t(\Y, \V)$ and in $G^\infty(\Y, \V)$ are finite in number and
coincide for large~$t$.
\end{lemma}

\begin{rem}
The monotonicity property implies that the indegree of a vertex
$\mathbf{x}
$ in $G^t(\Y, \V)$ converges almost surely to that in $G^\infty(\Y,
\V
)$ if the position of the vertex is zero, or in probability if its
position is nonzero. The lemma guarantees that there is actually always
almost sure convergence.
\end{rem}

\begin{pf}
We work conditionally on $\mathbf{x}=(x,s)\in\Y$, and start by
showing that
there exists an almost surely finite random variable~$M$ such that,
for all $t \in(0,\infty]$ and $\mathbf{y}\in\Y$ younger than
$\mathbf{x}$ and at
distance at least $M$ of $\mathbf{x}$, the vertices $\mathbf{x}$ and
$\mathbf{y}$ are not
linked in $G^t(\Y, \V)$.

The strategy is to find a coupling with a model independent of $t$,
based on the observation that the distance between $\mathbf{x}$ and
$\mathbf{y}$ in $\T
_t$ can be shortened by at most $2|x|$ compared to that in $\R$. Let
$K$ be the number of vertices in $\Y$ located at distance at most
$2|x|$ of $\mathbf{x}$, which is an almost surely finite random variable.
Consider the model where:
\begin{itemize}
\item the vertices within distance $2|x|$ of $\mathbf{x}$ are deleted;
\item the other vertices all come closer to $\mathbf{x}$ by distance $2|x|$;
\item the attachment rule $f$ is replaced by the rule $f_K\dvtx  i
\mapsto f(i+K)$.
\end{itemize}
%
%
It should be clear that the vertices $\mathbf{y}\in\Y$ younger than
$\mathbf{x}$, at
distance at least $2|x|$ of~$\mathbf{x}$,
which are linked to $\mathbf{x}$ in some finite graph $G^t(\Y, \V)$,
are also
linked to $\mathbf{x}$ in this model. Furthermore,
the indegree of $\mathbf{x}$ is still finite almost surely. Hence it suffices
to choose $M$ as the distance of $\mathbf{x}$ to the
furthest younger vertex it is linked to in this model, plus $2|x|$.

Finally, all that is left to show is that the incoming edges of
$\mathbf{x}$
linking it to a younger vertex $\mathbf{y}$ within distance $M$ coincide
in $G^t(\Y, \V)$ and in $G^\infty(\Y, \V)$, for large $t$. This follows
from the following two simple observations. First, the vertex $\mathbf
{x}$ is
linked to no
other younger vertex beyond distance $M$---in $G^\infty(\Y, \V)$ or in
any $G^t(\Y, \V)$---which could influence its indegree.
Second, for $t\ge|x|+M$, the vertices in $\Y$ and in $\Y^t$ within
distance $M$ of $\mathbf{x}$ coincide. Hence, for $t\ge|x|+M$,
the vertex $\mathbf{x}$ has the same incoming edges in $G^t(\Y, \V)$
and in
$G^\infty(\Y, \V)$.
\end{pf}

\subsection{Empirical outdegree distribution}\label{sec5.2} \label{subsectionoutdegree}
The following proposition descri\-bes what we know about the expected
outdegree distribution in the infinite picture.

%
%
\begin{prop}\label{limitingoutdegree}
For any $u \in(0,1]$, the expected outdegree distribution, defined by
the weights 
%
%
\begin{equation}
\label{limitingoutdegreeformula} \nu(k): = \P\bigl\{Y^\infty_u=k\bigr\}, \qquad
k\in\N\cup\{0,\infty\},
\end{equation}
is independent of $u$. Moreover, the measure $\nu$ is a probability
measure on $\N\cup\{0\}$ [i.e., $\nu(\infty)=0$] and it is light tailed
in the sense that for any $\delta\in(0,1-\gamma)$, we have
\[
\nu\bigl([k,\infty)\bigr) = o\bigl(e^{-k^\delta}\bigr).
\]
\end{prop}

\begin{pf}
The fact that $\nu(k)$ does not depend on $u$ is a simple consequence
of the rescaling invariance property. Therefore we only consider $u=1$,
and we watch for the law of $Y_1^\infty$, the outdegree of the point
$(0,1)$ in the infinite picture.
%

Attach to each vertex $\mathbf{x}\in\Y$ the value $\V_\mathbf{x}:=
\V((0,1), \mathbf{x})$.
Then each vertex can be identified with a point of $\R\times(0,1]
\times(0,1)$, and
the set of vertices becomes a Poisson point process of intensity one on
$\R\times(0,1]
\times(0,1)$. The idea is to define a domain $E_k$ such that the
probability that there is \emph{any} vertex in $E_k$ linked to $(0,1)$
is $O(e^{-k^\delta})$, and the probability that there are in total
\emph{at least $k$} vertices in the complement of $E_k$ (not necessarily
linked to 0) is also $O(e^{-k^\delta})$. This goes as follows:
\begin{itemize}
\item
Fix $\delta\in(0, 1-\gamma)$. Choose first $\gamma'\in[\gamma,
1-\delta)$ such that inequality~(\ref{linkingboundequation}) is
satisfied for some constant $c\in(0,\infty)$ (this is always possible,
in the linear case even with $\gamma'=\gamma$, see Corollary~\ref{linkingbound} and Remark~\ref{genrem}). Then, choose $\delta_1$,
$\delta_2$ such that $\delta<\delta_2<\delta_1< 1-\gamma'$.
\item
Introduce
\[
\hspace*{-10pt} E_k= \biggl\{(x,s,u) \in\R\times(0,1] \times(0,1) \dvtx
\frac{x} {\varphi^{-1} (u)} \ge k^{{\delta}/ {\delta_2}}, s \ge\biggl( \frac
{x} {\varphi^{-1} (u)}
\biggr)^{- ({1-\delta_1})/{\gamma
'}} \biggr\}.
\]
\end{itemize}
Then, from Corollary~\ref{linkingbound}, for any $\mathbf{x}=(x,s)$
and $u$
such that $(x,s,u)\in E_k$, we have
\begin{eqnarray*}
&& \P\bigl\{ (0,1) \underset t \leftrightarrow\mathbf{x} | \mathbf{x}\in
\Y,
\V_{\mathbf{x}
} =u \bigr\}
\\
&&\qquad \le  c\exp\biggl( -\frac{|x| s^{\gamma'}} {8 \gamma'
\varphi
^{-1}(u)} \biggr) \le c
\exp\biggl( - \frac{1} {8 \gamma'} \biggl(\frac{|x|} { \varphi
^{-1}(u)} \biggr)^{\delta_1}
\biggr). 
\end{eqnarray*}
%
Therefore, we get
\begin{eqnarray*}
&& \mathbb E \bigl[ \bigl| \bigl\{(x,s, u) \in E_k \dvtx  (x,s) \in\Y,
\V_{\mathbf{x}}=u, (0,1) \underset{\infty}\leftrightarrow(x,s) \bigr\}
\bigr| \bigr]
\\
&&\qquad \le\int\!\!\!\int\!\!\!\int_{E_k} \mathrm{d}x \,\mathrm{d}s
\,\mathrm{d}u\, c\exp\biggl( - \frac{1} {8 \gamma'} \biggl(\frac{|x|} {
\varphi
^{-1}(u)}
\biggr)^{\delta_1} \biggr)
\\
&&\qquad \le\int\!\!\!\int_{ \{{|x|}/{( \varphi^{-1}(u))} \ge
k^{\delta/{\delta_2} } \}} \mathrm{d}x \,\mathrm{d}u\,c\exp
\biggl( - \frac{1} {
8 \gamma'} \biggl(\frac{|x|} { \varphi
^{-1}(u)} \biggr)^{\delta_1}
\biggr)
\\
&&\qquad \le2 \int_0^1 \varphi^{-1}(u)\,\mathrm{d}u \int_{ [k^{
\delta/
{\delta_2} }, \infty)} c\exp\biggl( - \frac{1} {8 \gamma'}
y^{\delta_1} \biggr)\,\mathrm{d}y, 
\end{eqnarray*}
with the change of variable $y=|x|/\varphi^{-1}(u)$. 
The first integral is equal to the integral of $\varphi$ on $[0,\infty
)$, that is, $1/2$. For the second integral, introduce an appropriate
constant $C_1$ and get
\begin{eqnarray*}
&& \int_{ [k^{\delta/{\delta_2} }, \infty)} c\exp\biggl( - \frac{1}
{8 \gamma'}
y^{\delta_1} \biggr)\,\mathrm{d}y
\\
&&\qquad \le\int_{ [k^{\delta/{\delta_2}
}, \infty)}
C_1 \frac{\delta_2} {8 \gamma'} y^{\delta_2-1} \exp\biggl(-
\frac{1} {8
\gamma'} y^{\delta_2}\biggr)\,\mathrm{d}y \le C_1 \exp
\bigl(-k^{\delta}\bigr).
\end{eqnarray*}
The right-hand side is a bound to the expected number of vertices in
$E_k$ linked to $(0,1)$,
and thus it is also a bound to the probability that there is any
vertex in $E_k$ linked to $(0,1)$.%

Now, with an easier calculation we get that the total Lebesgue measure
of the complement of $E_k$ is bounded by
\begin{eqnarray*}
&& \int\!\!\!\int\!\!\!\int_{\R\times(0,1] \times(0,1)}  \mathrm{d}x
\,\mathrm{d}u\,\mathrm{d}s \bigl( \un\bigl\{y\le k^{\delta/{\delta_2}}\bigr\} + \un
\bigl\{s \le
y^{-({1-\delta_1})/ {\gamma'}}\bigr\} \bigr)
\\
&&\qquad \le2 \int_0^1 \varphi^{-1}(u)\,
\mathrm{d}u \int_{(0,\infty)} \bigl(\un\bigl\{y\le k^{\delta/{\delta_2}}
\bigr\} + (1\wedge y)^{-(1-\delta_1)/
{\gamma
'}} \bigr) \,\mathrm{d}y
\end{eqnarray*}
and is therefore less than $k^{{\delta}/ {\delta_2}}$ plus a
constant $C_2$. As the total number of points of~$\Y$ in this domain is
a Poisson variable
of parameter less than $k^{{\delta}/ {\delta_2}}+C_2$, we have
\begin{eqnarray*}
&& \P\bigl\{ \bigl| \bigl\{(x,s, u) \notin E_k \dvtx (x,s) \in\Y, \V
_{\mathbf{x}
}=u, (0,1) \underset{\infty} \leftrightarrow(x,s) \bigr\} \bigr| \ge k
\bigr\}
\\
&&\qquad \le\frac{(k^{{\delta}/ {\delta_2}} + C_2)^k} {k!} \le\frac{1} {
\sqrt{2 \pi k}} \biggl(\frac{e} k
\bigl(k^{{\delta}/ {\delta
_2}}+C_2\bigr) \biggr)^k,
\end{eqnarray*}
by Stirling's formula. As $\delta<\delta_2$ the right-hand side is
decaying superexponentially fast and therefore, summing up the estimates,
the overall probability that the outdegree of $(0,1)$ is greater than
or equal to $k$ is bounded by a~constant multiple of $\exp(-k^{\delta
})$. Hence $\nu([k,\infty))= O(\exp(-k^{\delta}))$, as claimed.
\end{pf}

The same proof, with the sets $E_k$ and their complements replaced by
their restrictions to $x \in(-t/2, t/2]$ also yields
%
%
\begin{equation}
\label{outdegreedomination} \P\bigl\{Y^t_u \ge k\bigr\}
\le(C_1 + C_3) \exp\bigl(-k^{\delta}\bigr)
\end{equation}
with the \emph{same constants} $C_1$ and $C_3$ for any $u$ and $t$.
Hence, the variables $Y^t_u$ are stochastically dominated by a
light-tailed random variable (this variable may not be $Y^\infty_1$,
recall that $Y^t_u$ is not monotone in $t$).

Take $g$ a function satisfying $g(k)=O(\exp(k^{\delta}))$ for some
$\delta<1-\gamma$, and define 
\[
\xi_t (\mathbf{x},\setZ, \W):= g \bigl(Y_\mathbf{x}\bigl(
G^t\bigl(\setZ\cup\{\mathbf{x}\}, \W\bigr)\bigr) \bigr),
\]
%
for $t\in(0,\infty]$, so that $\xi_t (\mathbf{x},\Y, \V)=
g(Y^t_\mathbf{x})$.
Domination~(\ref{outdegreedomination})
provides the uniform moment condition (for any given $p>1$).
Theorem~\ref{theooutdegree} follows, provided we prove the convergence
in probability of
$\xi_t ((0,u),\Y, \V)$ to $\xi_{\infty}((0,u),\Y, \V)$, for any
$u\in
(0,1]$. The following lemma proves more, and also
completes the proof of Proposition~\ref{propositionlocalconvergence}.

%
%
\begin{lemma}\label{lemmalocalconvergenceoutdegree}
Almost surely, for any $\mathbf{x}=(x,s)\in\Y$, the outgoing edges
of $\mathbf{x}$ in
$G^t(\Y, \V)$ and in $G^\infty(\Y, \V)$ are
finite in number and coincide for large~$t$.
\end{lemma}

\begin{pf}
Again, we suppose without loss of generality $s=1$ and work
conditionally on $\mathbf{x}=(x,1)\in\Y$.
Observe that if $M$ is any finite number then, almost surely, all the indegrees
of vertices in the graph $G^t(\Y, \V)$ with spatial position in $[x-M,
x+M]$ go to the
corresponding indegrees in $G^\infty(\Y, \V)$. Therefore, almost
surely, the outgoing edges linking $\mathbf{x}$ to a vertex within
distance $M$
of $\mathbf{x}$ coincide in $G^t(\Y, \V)$ and in $G^\infty(\Y, \V
)$, for large
$t$. The latter remains true if $M$ is random, but finite almost
surely. The lemma then follows if we show that there exists an almost
surely finite random variable $M$ such that for all $t \in(0,\infty]$,
for each $\mathbf{x}' \in\Y$ at distance at least $M$ of $\mathbf
{x}$, the vertices $\mathbf{x}
$ and $\mathbf{x}'$ are not linked in $G^t(\Y, \V)$.

To prove this, we use again the coupled model introduced in the proof
of Lemma~\ref{lemmalocalconvergenceindegree}. Again, the vertices
linked to $\mathbf{x}$ in some finite graph $G^t(\Y, \V)$ are also
linked to $\mathbf{x}
$ in the coupled model. Furthermore, in the coupled model, it is clear
that the \emph{outdegree} of $\mathbf{x}$ is still finite almost
surely, and we
can simply choose $M$ to be the distance of $\mathbf{x}$ to the
furthest vertex
it is linked to in this model, plus $2|x|$.
\end{pf}

\subsection{Clustering}\label{sec5.3}
\subsubsection{Average clustering coefficient}\label{sec5.3.1}
In this part, consider, for $t\in(0,\infty]$, the functionals $\xi_t$
and $\xi_t'$ defined by
\begin{eqnarray*}
\xi_t(\mathbf{x}, \setZ, \W)&=& c^{\mathrm{loc}}_\mathbf{x}
\bigl(G^t\bigl(\setZ\cup\{\mathbf{x}\}, \W\bigr)\bigr),
\\
\xi_t'(\mathbf{x}, \setZ, \W)&=& \un{ \bigl\{\mathbf{x}
\in V_2 \bigl(G^t\bigl(\setZ\cup\{\mathbf{x} \}, \W\bigr)
\bigr) \bigr\}},
\end{eqnarray*}
with the convention $ \xi_t(\mathbf{x}, \setZ, \W)=0$ if $\mathbf
{x}\notin V_2
(G^t(\setZ\cup\{\mathbf{x}\}, \W) )$, that is, if $\mathbf
{x}$ has degree less
than two. Thanks to Proposition~\ref{propositionlocalconvergence} and
its corollary, we know that for any $\mathbf{x}$, there is almost sure
convergence of
$\xi_t(\mathbf{x}, \Y, \V)$ to $\xi_\infty(\mathbf{x}, \Y, \V
)$, and of $\xi_t'(\mathbf{x}, \Y,
\V)$ to $\xi_\infty'(\mathbf{x}, \Y, \V)$. In particular,
condition \textup{(A)} of
Theorem~\ref{theoLLN} is satisfied for both functionals. Moreover, as
they take
values in $[0,1]$, the uniform moment condition~\textup{(B)} is also
satisfied. We immediately deduce the convergence in $L^1$ and in
probability of
\[
\frac{1} t \sum_{\mathbf{x}\in\Y^t} c^{{\mathrm{loc}}}_{\mathbf
{x}}
\bigl(G^t(\Y, \V)\bigr)\quad\mbox{and}\quad\frac{\llvert V_2 \rrvert
} t
\]
to the constants $ \mathbb E[\xi_\infty((0,U), \Y, \V)]$ and $\P\{
(0,U) \in
V_2 (G^t(\Y\cup\{(0,U)\}, \V) )\}$, respectively. Hence the
average clustering coefficient converges in probability to
\[
c^{\mathrm{av}}_\infty:= \mathbb E \bigl[\xi_\infty\bigl((0,U),
\Y, \V\bigr) | (0,U) \in V_2 \bigl(G^t\bigl(\Y\cup\bigl
\{(0,U)\bigr\}, \V\bigr) \bigr) \bigr].
\]
This constant is the expected local clustering coefficient of the
infinite graph at vertex $(0,U)$, conditionally on the event that its
degree is at least two. It is hard to compute analytically, but it
clearly belongs to $(0,1)$. The first part of Theorem~\ref
{theoclustering} is proved.

\subsubsection{Global clustering coefficient}\label{sec5.3.2}
The estimation of the global clustering coefficient relies on separate
estimations of the number of triangles and of the number of open
triangles in the network. 
We choose to count the triangles \emph{from their youngest vertex}, and
define the functional
$ \xi_t(\mathbf{x}, \setZ, \W) $
to be the number of triangles in $G^t(\setZ\cup\{\mathbf{x}\}, \W)$
having $\mathbf{x}
$ as youngest vertex. Again, Proposition~\ref
{propositionlocalconvergence} ensures that condition~\textup{(A)} is
satisfied. The simple observation that
$\xi_t(\mathbf{x}, \Y, \V)$ is bounded from above by $Y^t_\mathbf
{x}(Y^t_\mathbf{x}-1)/2$,
together with inequality~(\ref{outdegreedomination}),
ensures that the uniform moment condition~\textup{(B)} is satisfied for any
$p>1$, and we can apply Theorem~\ref{theoLLN}.
The number of triangles in the network $G^t(\Y, \V)$, divided by $t$,
converges to a positive and finite constant.
In other words, the number of triangles is asymptotically {proportional
to the number of vertices}.

Similarly, we introduce the functionals
\[
\xi_t'(\mathbf{x}, \Y, \V)=\frac{Z^t_\mathbf{x}(1) (Z^t_\mathbf
{x}(1)-1)} 2
\]
and
\[
\xi_t''(\mathbf{x}, \Y,
\V)=Y^t_\mathbf{x}Z^t_\mathbf{x}(1)+
\frac
{Y^t_\mathbf{x}(Y^t_\mathbf{x}-1)}2,
\]
%
where $\xi_t'$ corresponds to the open triangles whose tip $\mathbf
{x}$ is the
oldest vertex, and $\xi_t''$ are the remaining open triangles with tip
in~$\mathbf{x}$.
For both functionals, condition~\textup{(A)} follows again from Proposition~\ref
{propositionlocalconvergence}.
Condition~\textup{(B)} for functional $\xi_t''$ is also automatically satisfied,
for any $1<p<\frac{1}\gamma$. More precisely, to bound the expectation
of the product $(Y^t_U Z^t_U(1))^p$, first use their independence
conditionally on $U=u$, then use the domination~(\ref
{outdegreedomination}) to bound uniformly $\mathbb E[(Y_u^t)^p]$, before
integrating with respect to $u$. Therefore the number of open triangles
whose tip is not the oldest vertex, divided by $t$, converges in
probability to a positive and finite constant.

It is only for the functional $\xi'_t$ that we must discuss different
cases. Suppose first $\sum k^2 \mu(k)=\infty$, which implies $\mathbb
E [\xi
'_\infty( \Y, \V) ] = \infty$. In that case, Theorem~\ref{theoLLN}
and Remark~\ref{remLLN} imply that the number of open triangles with
tip the oldest vertex, divided by $t$, goes to $+\infty$ in
probability. Hence, the global clustering coefficient converges in
probability to zero.
Finally, suppose $\sum k^2 \mu(k)<\infty$ and hence $\mathbb E
[\xi'_\infty
( \Y, \V) ] < \infty$. The monotonicity property implies that the
variables $\xi'_t((0,U),\Y^t, \V^t)$ are always uniformly integrable,
even when condition~\textup{(B)} is not satisfied,\footnote{If $\gamma<\frac{1}2$,
then \textup{(B)} holds for any $1<p<\frac{1}{2\gamma}$, but if $\gamma=\frac{1}2$
and $\sum k^2 \mu(k)<\infty$, then \textup{(B)} does not hold.} and allows to
conclude that the global clustering coefficient converges in
probability to a positive constant.

\subsection{Empirical edge length distribution}\label{sec5.4}
The law of the distribution $\lambda_t$, the \emph{rescaled} empirical
edge length distribution in the original graph $G_t$, is the same as
the law of the \emph{unrescaled} empirical edge length distribution in
the graph $G^t(\Y, \V)$, which we will denote by $\tilde\lambda
_t$. We have,
abbreviating $E^t:=E(G^t(\Y^t, \V^t))$ and assuming it is not empty,
\[
\tilde\lambda_t = \frac{1} {|E^t|} \sum
_{(\mathbf{x}', \mathbf
{x}) \in E^t} \delta_{d(\mathbf{x}',
\mathbf{x})} = \biggl(\sum
_{\mathbf{x}\in\Y^t} Y^t_\mathbf{x}
\biggr)^{-1} \sum_{\mathbf{x}\in\Y^t} \mathop{\sum
_{\mathbf{x}' \in\Y^t,
\mathbf{x}' \mathop{\leftrightarrow}\limits_{t}\mathbf{x}}}_
{\mathbf{x}'\ \mathrm{older\ than}\ \mathbf{x}} \delta_{d(\mathbf{x}',
\mathbf{x})},
\]
where we have chosen to count each edge from its younger vertex. Define
the probability measure $\lambda$ on $[0,+\infty)$ by
\[
\lambda(A) = \frac{1} {\mathbb E[Y^\infty_{(0,U)}]} \mathbb E \bigl[ \bigl|
\bigl\{ (x,s) \in\Y
\dvtx (x,s) \underset\infty\leftrightarrow(0,U), s<U, |x| \in A \bigr
\} \bigr| \bigr],
\]
for any Borel set $A\subset[0,\infty)$, where $U$ denotes a random
variable uniformly distributed on $(0,1)$ and independent of $\Y$ and
$\V$.
By application of Theorem~\ref{theoLLN} we get, for any $x\in
[0,\infty)$,
\[
\tilde\lambda_t\bigl([x,\infty)\bigr) \longrightarrow\lambda
\bigl([x,\infty)\bigr),
\]
in\vspace*{1pt} probability. A technical but simple argument shows convergence in
probability of $\tilde\lambda_t$ to $\lambda$ in the space of probability
measures on $[0,+\infty)$,
equipped with the L\'evy--Prokhorov metric, which defines narrow convergence.
This proves the first part of Theorem~\ref{theoedgelength}.

Next we estimate the order of $\lambda([K,\infty))$ when $K$ is large.
Fix $K>0$. We have
\[
\lambda\bigl([K,\infty)\bigr) = 2 \int_\Omega\,\mathrm{d}x
\otimes\mathrm{d}t \otimes\mathrm{d}u \otimes\mathrm{d}s\,\P\bigl\{(x,s)
\underset\infty\leftrightarrow(0,t) | \V\bigl((x,s), (0,t)\bigr) =u
\bigr\},
\]
where $\Omega$ is the domain $ \{(x, t, u, s) \in[K,\infty)
\times
(0,1)^3 \dvtx  s<t \}$. The factor two comes from the fact that
we have chosen $x>0$. The linking probability contains an implicit
conditioning on the event that $(x,s)$ and $(0,t)$ are in $\Y$.
As in the proof of Corollary~\ref{linkingbound} we can rewrite
\begin{eqnarray*}
\P\bigl\{(x,s) \underset\infty\leftrightarrow(0,t) | \V\bigl((x,s), (0,t)
\bigr) =u \bigr\} &=& \P\bigl\{Z_s^\infty(t) \ge
f^{-1} \bigl( t x / \varphi^{-1}(u) \bigr) \bigr\}
\\
&=& \P\bigl\{Z_{s/t}^\infty(1) \ge f^{-1} \bigl( t
x / \varphi^{-1}(u) \bigr) \bigr\},
\end{eqnarray*}
where $f^{-1}$ is the right-continuous inverse of $f$. Changing the variable
\[
(x,t,u,s) \mapsto(y,z,u, r)\qquad\mbox{with } y=\frac{t x} {\varphi
^{-1}(u)}, z=
\frac{K \varphi^{-1}(u)} x, r=\frac{s} t,
\]
sending $\Omega$ to $\Omega'= \{(y,z,u,r) \in(0, \infty)^3 \times
(0,1), z \le\frac{K} y, u\le\varphi(z)\}$ we get
\begin{eqnarray*}
\lambda\bigl([K,\infty)\bigr) &=& 2 K^{-1} \int_{(0,\infty)}
\mathrm{d}y\, y \biggl(\int_{(0,{K}/ y)} \mathrm{d}z \int
_{(0, \varphi(z))} \mathrm{d}u\,\varphi^{-1}(u) \biggr)
\\
&&{}\times  \biggl(
\int_{(0,1)} \mathrm{d}r\, \P\bigl\{Z_r^\infty(1)
\ge f^{-1}(y) \bigr\} \biggr)
\\
&=& 2 K^{-1} \int_0^\infty\mathrm{d}y\, y
I \biggl(\frac{K} y \biggr) J(y),
\end{eqnarray*}
with $I$ and $J$ defined to be the two integrals in brackets in the
first line.
For an estimate of~$J$, we simply note that
$J(y)= \mu(\lceil f^{-1}(y)\rceil, \infty) \asymp1\wedge
y^{-1/\gamma}$.
For an estimate of~$I$ we start with the equality
\[
\int_{(0, \varphi(z))} \mathrm{d}u\,\varphi^{-1}(u)= \int
_{(0,\infty)} \varphi(z \vee v) \,\mathrm{d}v,
\]
based on the observation that they both represent the area of
\[
\bigl\{(u,v) \in(0,\infty)^2 \dvtx  u \le\varphi(z), v\le\varphi(z)
\bigr\},
\]
to get
\begin{eqnarray*}
I(a) &=& \int_{(0,a)\times(0,\infty)} \mathrm{d}z \otimes\mathrm{d}v\,
\varphi(z
\vee v)
= 2 \int_{(0,a)} v \varphi(v) \,\mathrm{d}v +a \int
_{(a,\infty)} \varphi(v) \,\mathrm{d}v.
\end{eqnarray*}

\noindent Now, elementary calculations yield
\[
I(a) \asymp\cases{ a \wedge1, &\quad if $\displaystyle\int_0^\infty
v \varphi(v) \,\mathrm{d}v<\infty$,
\cr
a \wedge a^{2-\delta}, &\quad if $
\displaystyle\varphi(v) \asymp1 \wedge v^{-\delta}$ for $\delta\in(1,2]$.}
\]
Finally, another elementary calculation shows that we have
\[
\lambda\bigl([K,\infty)\bigr) \asymp1 \wedge\bigl(K^{-1}+
K^{1-({1}/ \gamma)} + K^{1-\delta}\bigr),
\]
and Theorem~\ref{theoedgelength} follows.

\section{Variants of the model}\label{sec6}

\subsection{Discrete versus continuous time}\label{sec6.1}
We have decided to define our model in continuous time, as this is
naturally aligned with our techniques of proof. We expect that
all our results hold without change for the analogous discrete model,
but we have not attempted to derive this from our results as
we do not expect to get interesting insights from this. We point out
that the weak law of large numbers in~\cite{YukichPenrose}
includes a de-Poissonisation, but this cannot be applied directly in
our case as it does not deal with the explicit time dependence of
the attachment probabilities.%

\subsection{The case \texorpdfstring{$\gamma\ge1$}{$gamma>=1$}}\label{sec6.2}
This assumption leads to a very different behavior, which we briefly
discuss. Lemma~\ref{polynomialgrowth} does not hold anymore. Instead, the indegree of a
fixed vertex (the oldest one, e.g.), grows roughly linearly, and it
will be eventually connected to a positive proportion of the younger
vertices. The length of its incoming edges is thus of order one.
The law of large numbers, Theorem~\ref{theoLLN}, holds unchanged, as well as
Theorem~\ref{theoconvergenceindegree}. That said, we have $\sum k \mu(k)=\infty$, which implies
that the total number of edges is superlinear. The empirical outdegree
distribution converges
vaguely to the null distribution, as all the mass escapes to infinity.
In the infinite picture, the outdegree of each vertex is almost surely infinite.
Finally, the same phenomenon happens to the empirical edge length
distribution, if we still rescale it by the same factor of $t$. Note
that~\cite{Aiello}
also contains results for the case $\gamma=1$, corresponding to
$pA_1=1$ in their notation, which are consistent with our observations.%

\subsection{Higher-dimensional space}\label{sec6.3}
We have chosen to present our results for spatial distributions given
as uniform distributions on the one-dimensional torus to keep
technicalities to a minimum. Nothing would change if we replace the
torus by the unit interval, as boundary effects\vspace*{1pt} will be negligible.
There is also no problem generalizing results to higher-dimensional
tori~$\T^d$, or unit cubes. In fact, if we connect
the vertex~$\mathbf{y}=(y,t)$ to an older vertex~$\mathbf{x}$ with probability
\[
\varphi\biggl(\frac{t^{1/d}d(\mathbf{x},\mathbf{y})}{f(Z_\mathbf
{x}(t-))^{1/d}} \biggr),
\]
and normalize the profile function so that
\[
\int_{\R^d} \varphi\bigl( d(0,y) \bigr) \,\mathrm{d}y=1,
\]
we can recover Theorems~\ref{theoconvergenceindegree}, \ref{theooutdegree} and~\ref{theoclustering} verbatim by the same arguments. In the
empirical edge length distribution we need to
rescale by a factor of $t^{{1}/ d}$ instead of $t$, and we obtain a
limiting edge length distribution~$\lambda$, which
depends on the dimension. If the profile function scales like $\varphi
(x) \asymp1 \wedge x^{-\delta}$
we need to have $\delta>d$ to meet the integrability condition.
Then we recover Theorem~\ref{theoedgelength} with $\eta\in(0,d]$ the smallest of the
three constants $d$, $\delta- d$ and $d(\frac{1}\gamma-1)$.
If $\eta>1$, then $\lambda$ has a first moment, and the mean edge
length is of order $t^{-{1}/d}$.%

\subsection{More general underlying spaces}\label{sec6.4}
It is no problem to define our model in a general metric space. However
this can lead to a significant
change in the behavior, as inhomogeneities in the underlying space
introduce an element of fitness of
individual vertices. In a similar spirit one can change the spatial
distribution of \mbox{incoming} vertices.
Again one would expect that small changes do not change the qualitative
behavior, whereas highly
fluctuating densities can have a major effect. These problems have
recently been discussed by
Jordan~\cite{Jordan-new} for a closely related model.

\subsection{Further remarks and problems}\label{sec6.5}
Our technique allows the analysis of a wide range of functionals of
spatial preferential attachment
networks, and we have only picked those that appeared most interesting
to us at this point. Other
network ``metrics'' that could be studied are the total edge length,\vadjust{\goodbreak}
the number of occurrences
of a particular finite subgraph (or motif), or the number of (suitably
defined) high density spots.%

More generally, the local limit results established here offer a handle
to the study of global connectivity
problems, for example, the existence and diameter of a giant component.
This would be of particular
interest as nontrivial rigorous results on the existence of the giant component
have never been established for dynamic network models that are not
locally tree-like.
Existence of a giant component for an interesting static example, which
is not locally tree-like, is
studied in~\cite{Bollobas}. A~first discussion including a
simulation-based conjecture for the location of a phase transition related
to the existence of a giant component in the model of~\cite{Aiello} can
be found in~\cite{Cooper}.

\section*{Acknowledgments}
We would like to thank two anonymous referees for their careful reading of the
manuscript and for suggesting several improvements.


%
%

\printaddresses

\end{document}